\documentclass[12pt,a4paper]{amsart}

\usepackage{amsfonts,amscd,amsmath,amssymb,amsthm,mathrsfs}
\usepackage{array,latexsym,float,enumerate,ifpdf}
\usepackage{graphicx,epsfig}
\usepackage[top=2.7cm, bottom=2.7cm, left=2.5cm, right=2.5cm,twoside=false]{geometry}
\usepackage[all]{xy}
\usepackage[usenames,dvipsnames]{color}

\makeatletter
\def\subsection{\@startsection{subsection}{3}%
  \z@{.5\linespacing\@plus.7\linespacing}{.5\linespacing}%
  {\normalfont\itshape}}
\makeatother

\makeatletter
\renewenvironment{proof}[1][\proofname]{%
  \par\pushQED{\qed}\normalfont%
  \topsep6\p@\@plus6\p@\relax
  \trivlist\item[\hskip\labelsep\bfseries#1\@addpunct{.}]%
  \ignorespaces}{%
  \popQED\endtrivlist\@endpefalse}
\makeatother

\numberwithin{equation}{section}
\theoremstyle{plain}
\newtheorem{theorem}{Theorem}[section]
\newtheorem{lemma}[theorem]{Lemma}
\newtheorem{proposition}[theorem]{Proposition}
\newtheorem{corrolary}[theorem]{Corollary}
\newtheorem{conjecture}[theorem]{Conjecture}
\theoremstyle{definition}
\newtheorem{definition}[theorem]{Definition}
\newtheorem{example}[theorem]{Example}
\newtheorem{noname}[theorem]{}
\newtheorem{remark}[theorem]{Remark}
\newtheorem{construction}[theorem]{Construction}
\newtheorem{notation}[theorem]{Notation}
\theoremstyle{remark}
\newtheorem*{smallremark}{Remark}
\newtheorem*{claim2}{Claim}
\newtheorem{case}{Case} \makeatletter \@addtoreset{case}{theorem}\makeatother
\newtheorem*{claim}{Claim} % \makeatletter \@addtoreset{claim}{theorem}\makeatother

\newcommand{\bthm}{\begin{theorem}}
\newcommand{\bprop}{\begin{proposition}}
\newcommand{\blem}{\begin{lemma}}
\newcommand{\bcor}{\begin{corrolary}}
\newcommand{\brem}{\begin{remark}}
\newcommand{\bdfn}{\begin{definition}}
\newcommand{\bitem}{\begin{itemize}}
\newcommand{\bex}{\begin{example}}
\newcommand{\bno}{\begin{noname}}
\newcommand{\bsrem}{\begin{smallremark}}
\newcommand{\bnot}{\begin{notation}}
\newcommand{\bcon}{\begin{construction}}
\newcommand{\bca}{\begin{case}}
\newcommand{\bcl}{\begin{claim}}
\newcommand{\beq}{\begin{equation}}

\newcommand{\eeq}{\end{equation}}
\newcommand{\ecl}{\end{claim}}
\newcommand{\eca}{\end{case}}
\newcommand{\econ}{\end{construction}}
\newcommand{\enot}{\end{notation}}
\newcommand{\esrem}{\end{smallremark}}
\newcommand{\eno}{\end{noname}}
\newcommand{\eex}{\end{example}}
\newcommand{\eitem}{\end{itemize}}
\newcommand{\ethm}{\end{theorem}}
\newcommand{\eprop}{\end{proposition}}
\newcommand{\elem}{\end{lemma}}
\newcommand{\ecor}{\end{corrolary}}
\newcommand{\erem}{\end{remark}}
\newcommand{\edfn}{\end{definition}}
\newcommand{\benum}{\begin{enumerate}}
\newcommand{\eenum}{\end{enumerate}}

\newcommand{\wt}{\widetilde}
\newcommand{\cal}[1]{\mathcal{#1}}

\newcommand{\ds}{\displaystyle}

\def\8{\infty}
\def\.{\cdot}
\def\PP{\mathbb{P}}

\def\C{\mathbb{C}}

\def\Q{\mathbb{Q}}

\def\E{\widehat{E}}

\def\:{\colon}

\def\ssk{\smallskip}
\def\bsk{\bigskip}
\newcommand{\noin}{\noindent}
\def\Bk{\operatorname{Bk}}

\def\rk{\operatorname{rk}}

\def\Supp{\operatorname{Supp}}
\def\Pic{\operatorname{Pic}}
\def\dim{\operatorname{dim}}
\def\Exc{\operatorname{Exc}}

\def\Spec{\operatorname{Spec}}
\def\NS{\operatorname{NS}}

\def\core{\operatorname{core}}
\def\ind{\operatorname{ind}}
\def\vc{\operatorname{vc}}
\def\tip{\operatorname{tip}}

\ifpdf \usepackage[linkbordercolor={0 0 1}]{hyperref} \else \usepackage[hypertex,linkbordercolor={0 0 1}]{hyperref} \fi
\newcommand{\red}[1]{}

\begin{document}

\title[Cuspidal curves, minimal models and Zaidenberg's Finiteness Conjecture]{Cuspidal curves, minimal models and\\ Zaidenberg's Finiteness Conjecture}

\author[Karol Palka]{Karol Palka}
\address{Karol Palka: Institute of Mathematics, Polish Academy of Sciences, ul. \'{S}niadeckich 8, 00-656 Warsaw, Poland}

\thanks{The author was supported by the National Science Center, Grant No.\ 2012/05/D/ST1/03227 and by the Foundation for Polish Science within the Homing Plus programme, cofinanced by the European Union, Regional Development Fund.}

\email{palka@impan.pl}
\subjclass[2000]{Primary: 14H50; Secondary: 14J17, 14R05}
\keywords{Cuspidal curve, rigidity conjecture, Finiteness Conjecture, Cremona transformation, log Minimal Model Program, minimal model}

\begin{abstract}
  Let $E\subseteq \PP^2$ be a complex rational cuspidal curve and let $(X,D)\to (\PP^2,E)$ be the minimal log resolution of singularities. We prove that $\bar E$ has at most six cusps and we establish an effective version of the Zaidenberg Finiteness Conjecture (1994) concerning Eisenbud-Neumann diagrams of $E$. This is done by analysing the Minimal Model Program run for the pair $(X,\frac{1}{2}D)$. Namely, we show that $\PP^2\setminus E$ is $\C^{**}$-fibred or for the log resolution of the minimal model the Picard rank, the number of boundary components and their self-intersections are bounded.
\end{abstract}

\begin{center} \red{CORRECT 'structural'} \end{center}

\maketitle

\section{Main results}\label{sec:main result}

Let $\bar E\subseteq \PP^2$ be a complex rational curve having only locally analytically irreducible singularities (cusps\footnote{Cusps should not been confused with 'simple' or 'ordinary cusps', which are locally isomorphic to the singular point of $x^2-y^3$ at $0\in \Spec \C[x,y]$.}). It is an image of a morphism $\PP^1\to \PP^2$ which is \mbox{1-1} on closed points. Classifying such curves up to a choice of coordinates on $\PP^2$ is a long-standing open problem with many connections, see \cite{MaSa-cusp}, \cite{Orevkov}, \cite{OrZa_cusp}, \cite{FLZa-cusp}, \cite{Fenske_cusp}, \cite{Tono-on_the_number_of_cusps}, \cite{FLMN}, \cite{BoroLivi-HeegaardFloer_and_cusps}. We shall further assume that $\kappa(\PP^2\setminus \bar E)=2$, otherwise there is a classification. Many conjectures have been made. For instance, the Coolidge-Nagata conjecture predicts that $\bar E$ is Cremona equivalent to a line. The weak rigidity conjecture of Flenner-Zaidenberg states that if $(X,D)\to (\PP^2,\bar E)$ is the minimal log resolution of singularities then the Euler characteristic of the logarithmic tangent sheaf of $(X,D)$ vanishes. The first conjecture is equivalent to the vanishing of $h^0(2K_X+E)$, where $E$ is the proper transform of $\bar E$ on $X$, and the second to the vanishing of $h^0(2K_X+D)$ (see \ref{con:conjectures}).

Suggested by these equivalences, in this article we show that in understanding the geometry of cuspidal curves the fundamental role is played by the divisor $K_X+\frac{1}{2}D$. We propose a new approach to the above problems based on the analysis of the Minimal Model Program run for the pair $(X,\frac{1}{2}D)$; see \ref{rem:strategy} for the discussion of the strategy. The key result is the description of the structure of a minimal model and its minimal log resolution, an \emph{almost minimal model}; see \ref{def:psi_and_models} for definitions and \ref{thm:MAIN1detailed} for a more detailed version of the result.

We write $b_i(-)$ for $\dim H^i(-,\Q)$ (the $i$-th Betti number) and $p_a$ for the arithmetic genus. Put $\kappa_{1/2}(\PP^2,\bar E):=\kappa(K_X+\frac{1}{2}D)$ and $p_2(\PP^2,\bar E):=h^0(2K_X+D)$. A \emph{$\C^{**}$-fibration} is a fibration whose general fiber is isomorphic to $\C^{**}=\C^1\setminus\{0,1\}$.

\bthm\label{thm:MAIN1} Let $\bar E\subseteq \PP^2$ be a complex rational cuspidal curve for which $\PP^2\setminus \bar E$ is of log general type and let $(X,D)\to (\PP^2,\bar E)$ be the minimal log resolution of singularities. Let $(X',\frac{1}{2}D')$ and $(Y,\frac{1}{2}D_Y)$ denote respectively an almost minimal and a minimal model as defined in \ref{def:psi_and_models}. Put $n=p_a(D')$. Then: \benum[(1)]

\item $X'\setminus D'\cong Y\setminus D_Y$ is isomorphic to an open subset of $\PP^2\setminus \bar E\cong X\setminus D$ with the complement being a sum of $n$ disjoint curves isomorphic to $\C^*$.

\item $n+p_2(\PP^2,\bar E)\leq 5$.

\item If $\kappa_{1/2}(\PP^2,\bar E)=-\8$ then one of the following holds: \benum[(a)]
    \item $Y$ admits a $\PP^1$-fibration with irreducible fibers inducing a $\C^{**}$-fibration on $\PP^2\setminus\bar E$.
    \item $(Y,\frac{1}{2}D_Y)$ is an open log terminal log del Pezzo surface of rank $1$, whose boundary has $n+1\leq 6$ components.
     \eenum

\item If $\PP^2\setminus \bar E$ does not admit a $\C^{**}$-fibration then $$\rho(X')\leq p_2(\PP^2,\bar E)+19-n.$$

    \eenum\ethm

If a $\C^{**}$-fibration of $\PP^2\setminus\bar E$ is induced from a $\PP^1$-fibration of some surface of Picard rank $1$ (as above) then we call it a \emph{structural} $\C^{**}$-fibration. \red{CORRECT: no base points.}

Tono \cite{Tono-on_the_number_of_cusps} bounded the number of maximal twigs of $D$ (for a definition of a twig see Sec.\ \ref{ssec:open surfaces}) by $17-p_2(\PP^2,\bar E)$. However, neither the number of components of $D$ contained in these twigs can be bounded from above nor their self-intersections can be bounded from below. Indeed, it is already so for the family of tricuspidal curves found by Flenner-Zaidenberg \cite[3.5]{FLZa-cusp}. Surprisingly, the situation changes after arriving at an almost minimal model.
\bthm\label{thm:MAIN2} Let $\bar E\subseteq \PP^2$ be a complex rational cuspidal curve, such that $\PP^2\setminus \bar E$ is of log general type and does not admit a structural $\C^{**}$-fibration. Let $(X',\frac{1}{2}D')$ be an almost minimal model as above. Then: \benum[(1)]

    \item $\rho(X')\leq 22,\ \#D'\leq 22 \text{\ \ and\ \ }K_{X'}\cdot D'\leq 16.$ The inequalities are strict in case $\kappa_{1/2}(\PP^2,\bar E)=-\8$.

    \item If $\kappa_{1/2}(\PP^2,\bar E)=-\8$ then the self-intersections of components of $D'$ are not smaller than $(-3)$. If $\kappa_{1/2}(\PP^2,\bar E)\geq 0$ then they are between $-40$ and $6$.
\eenum \ethm

\bsrem To see an immediate application of the above result assume  $\PP^2\setminus \bar E$ contains no $\C^*$. Then $n=0$, so $(X,D)=(X',D')$ and the bounds above are actually bounds on the log resolution, so we may in fact classify rational cuspidal curves for which $\PP^2\setminus \bar E$ contains no $\C^*$. In Corollary \ref{thm:MAIN3} below we show how to get around the hypothesis $n=0$ to obtain some unconditional bounds on the geometry of $D$.\esrem

Each reduced effective divisor can be decomposed as a sum of its maximal rational twigs plus the remaining part, called the \emph{core}. The above theorem coupled with an analysis of structural $\C^{**}$-fibrations in Section \ref{sec:Cstst} establishes an effective version of the Zaidenberg Finiteness Conjecture for cuspidal curves (see \ref{con:conjectures}(4) for the conjecture and \ref{def:core_and_EN-diagram} for the definition of the core graph).

\bcor\label{thm:MAIN3}(Effective Zaidenberg Finiteness Conjecture). Let $\bar E \subseteq \PP^2$ be a rational cuspidal curve whose complement is of log general type and let $(X,D)\to (\PP^2,\bar E)$ be the minimal log resolution of singularities. The core of $D$ has at most $20$ components and the core graph (hence the Eisenbud-Neumann diagram) has at most $31$ vertices. In particular, the sets of possible core graphs and Eisenbud-Neumann diagrams are finite.
\ecor

The following inequalities are of independent interest.

\bthm\label{thm:MAIN4} Let $\bar E\subset \PP^2$ be a rational cuspidal curve contained in the complex projective plane. Then $\bar E$ has at most six cusps. In fact, if $c$ is the number of cusps and $n$ is as above then $c+\max(0,2p_2(\PP^2,\bar E)+n-4,2p_2(\PP^2,\bar E)-3)\leq 6$. \ethm

Up to now the best known bound was $c\leq  8.5-\frac{1}{2}p_2(\PP^2,\bar E)<9$ \cite{Tono-on_the_number_of_cusps}. The bound $c\leq 8$ was obtained earlier in \cite{OrZa_cusp} provided $h^1(\cal T_X(-\log D))=0$. It was conjectured by Orevkov that in fact $c\leq 4$.

\bsk The tools built in this article play a basic role in our recent proof of the Coolidge-Nagata conjecture \cite{Palka-Coolidge_Nagata1}, \cite{Coolige-Nagata2}. They can be applied to study open surfaces of log general type.

\bsk\textsl{\textsf{Acknowledgements.}} The author is grateful to Mariusz Koras for discussions concerning producing $\C^*$'s contained in the affine part of the surface. His remark to reuse the BMY inequality at the final stage of the proof of Theorem \ref{thm:MAIN1}(4) (instead of using the bound obtained by Tono) led to improved numerical bounds in the inequalities. The author would also thank Mikhail Zaidenberg for discussing results in the literature.

\tableofcontents

\section{Preliminaries}\label{sec:preliminaries}

We recall some results from the theory of non-complete surfaces, also to settle notation. For a complete treatment the reader is referred to \cite{Miyan-OpenSurf}.

\subsection{Log surfaces and divisors}\label{ssec:open surfaces}
Given two $\Q$-divisors $T, T'$ we say that $T'$ is a \emph{subdivisor of $T$} ($T'\leq T$) if $T-T'$ is effective.  Let $T$ be a nonzero reduced divisor on a smooth complete surface $X$. If $R$ is a reduced subdivisor of $T$ we define $\beta_T(R)=R\cdot (T-R)$ and we call it a \emph{branching number of $R$ in $T$}. If $R$ is irreducible and nonzero we say that $R$ is a \emph{tip} or a \emph{branching component} if $\beta_T(R)\leq 1$ or $\beta_T(R)\geq 3$ respectively. We say that $T$ is an snc-divisor if its irreducible components are smooth and intersect transversally, at most two in one point. It is \emph{snc-minimal} if a contraction of any $(-1)$-curve in $T$ maps $T$ onto a divisor which is not snc. We call $T$ a \emph{chain} if it is a connected snc-divisor with linear dual graph. A (rational) \emph{twig} of $T$ is a (rational) snc-chain $R\leq T$ which contains no branching components of $T$ and contains a tip of $T$. A twig is \emph{maximal} if it is not properly contained in some other twig. Similarly maximal rational twigs are maximal in the sense of inclusion of supports in the set of rational twigs. We say that $T$ is a \emph{(rational) tree} if it is a connected snc-divisor (whose all components are rational and) such that the dual graph of $T$ contains no loops. A fork is a tree with exactly one branching component and three maximal twigs. The \emph{arithmetic genus of $T$} is $$p_a(T)=\frac{1}{2}T\cdot (K+T)+1,$$ where $K$ is the canonical divisor (class) on $X$. We often use the fact that for a rational tree $p_a(T)=0$. We denote the Iitaka-Kodaira dimension of the divisor $T$ (on $X$) by $\kappa(T)$. If $T=T_1+\ldots+T_k$ is a decomposition of a rational chain into irreducible components, such that $T_i\cdot T_{i+1}=1$ for $i<k$, then we write $T=[-T_1^2,-T_2^2,\ldots,-T_k^2].$ By $(m)_p$ we mean a sequence $(m,m,\ldots,m)$ of length $p$. By a curve we mean a one-dimensional variety. An \emph{$(n)$-curve} is a smooth rational curve with self-intersection $n$. A $(-1)$-curve which is a component of $T$ as above is called \emph{superfluous} if it intersects at most two other components of $T$, each at most once and transversally. A $(-2)$-twig is a twig such that all its irreducible components are $(-2)$-curves. A maximal $(-2)$-twig is a $(-2)$-twig which is not a proper subdivisor of another $(-2)$-twig. We define the \emph{discriminant of $T$} as $d(T)=\det(-Q(T))$, where $Q(T)$ is the intersection matrix of $T$. We put $d(0)=1$.

Assume now that $T$ is a fixed connected snc-divisor with no superfluous $(-1)$-curves and with intersection matrix which is not negative definite. (This is the case when $T$ is an snc-minimal boundary of an affine surface). Assume also that the intersection matrices of all its maximal twigs are negative definite (this is the case if $\kappa(K+T)=2$, \cite[6.13]{Fujita-noncomplete_surfaces}). For every twig $R$ of $T$ we define $$\ind (R)=\frac{d(R-\tip(R))}{d(R)} \text{\ \ and\ \ } \delta(R)=\frac{1}{d(R)},$$ where $\tip(R)$ denotes the unique tip of $T$ contained in $R$. The former number is usually called the \emph{inductance} or \emph{capacity} of $R$. We define $\Bk_T R$, the \emph{bark of $R$ with respect to $T$}, as the unique $\Q$-divisor supported on $\Supp R$, such that $$\Bk_T R\cdot R_0=\beta_T(R_0)-2$$ for every component $R_0$ of $R$, equivalently that $\Bk_T R\cdot R_0$ equals $-1$ if $R_0=\tip(R)$ and equals $0$ otherwise. Let now $T_i$, $i=1,\ldots,t$, be the maximal twigs of $T$. We put $$\ind (T)=\sum_{i=1}^t \ind (T_i)\text{\ \ and \ } \delta(T)=\sum_{i=1}^t\delta(T_i) ,$$ and we define the \emph{bark of $T$} ad $\Bk T=\Bk_T T_1+\ldots+\Bk_T T_t$.

If $\kappa(K+T)\geq 0$ then we have the Fujita-Zariski decomposition $$K+T=(K+T)^++(K+T)^-,$$ where $(K+T)^+$ is numerically effective and $(K+T)^-$ is effective, either empty or having a negative definite intersection matrix. Moreover, $(K+T)^+\cdot B=0$ for any curve $B$ contained in $\Supp (K+T)^-$.

\blem\label{lem:Bk} Let $T_i$ be a twig of a rational tree $T$ as above and let $T_0$ be a component of $T_i$. Denote the coefficient of $T_0$ in $\Bk T$ by $t_0$. Then \benum[(i)]

\item $0<t_0<1$.

\item If $T_0$ is the component meeting $T-T_i$ then $t_0=\delta(T_i)$.

\item $(\Bk T)^2=-\ind (T)$.

\item If there is no $(-1)$-curve $A$ on $X$, for which $T\cdot A\leq 1$, then $(K+T)^-=\Bk T$.

\eenum \elem

\begin{proof} Write $T_i=T_{i,1}+T_{i,2}+\ldots+T_{i,k_i}$, where $T_{i,j}$ are irreducible and $T_{i,j}\cdot T_{i,j+1}=1$ for $j<k_i$. Then by \cite[2.3.3.4]{Miyan-OpenSurf} the coefficient of $T_{i,j}$ in $\Bk T$ equals $d(T_{i,j+1}+\ldots+T_{i,k_i})/d(T_i)$. This gives (i), (ii) and (iii). Part (iv) follows from 2.3.11 loc.\ cit.\end{proof}

Let $\sigma\:(X',D')\to (X,D)$, where $D'=\sigma^{-1}_*D+\Exc \sigma$, be a blowup with a center on $D$ ($\sigma_*^{-1}$, or rather $(\sigma^{-1})_*$, denotes taking the proper transform). Write $\sigma^*D=D'+\mu\Exc \sigma.$ We say that $\sigma$ is \emph{inner (outer)} for $D$ if $\mu=1$ ($\mu=0$ respectively). If $D$ has smooth components we may equivalently ask that the center of $\sigma$ belongs to exactly two (one) components of $D$.

By a \emph{log surface} we mean a pair $(Y,B)$ consisting of a projective normal surface $Y$ together with an effective $\Q$-divisor, which can be written as $B=\sum b_iB_i$, where $B_i$ are distinct irreducible components and $0<b_i\leq 1$. It is (log) smooth if $X$ is smooth and $B$ is an snc-divisor.

\bdfn\label{def:resolution} Let $(Y,B)$ be a log surface and let $\pi\:(X,D)\to (Y,B)$ be a proper birational morphism from a log surface such that $X$ is smooth and $D=\pi_*^{-1}B+\Exc \pi$. We say that $\pi$ is a \emph{weak (embedded) resolution of singularities} if $\pi_*^{-1}B$ (the proper transform) is an snc-divisor. It is a \emph{log resolution} if $D$ is an snc-divisor (equivalently, $(X,D)$ is a smooth log surface).\edfn

\bsrem The role of maximal twigs and barks in the theory of log surfaces can be seen easily from the point of view of the logarithmic Minimal Model Program. For example, with the above assumptions and notation by the adjunction formula $$(K_X+T)\cdot T_0=\beta_{T}(T_0)-2<0,$$ so the tip $\tip(T_i)$, and by induction the whole maximal twig $T_i$, is contracted in the process of minimalization of $(X,T)$. Now if $\alpha\:(X,T)\to (X',\alpha_*T)$ is the contraction of maximal twigs then $$\alpha^*(K_{X'}+\alpha_*T)=K_X+T-\Bk T.$$ \esrem

Recall that by definition components of a twig of any divisor are smooth.

\bdfn\label{def:core_and_EN-diagram} Let $B=\sum B_i$ be a reduced effective divisor on a smooth projective surface. Assume $B$ has a connected support. \benum[(i)]

\item The \emph{(weighted) dual graph} of $B$ is a graph with a vertex (of weight $B_i^2$) for each irreducible component of $B_i$ and $B_i\cdot B_j$ edges between $B_i$ and $B_j$ for $i\neq j$.

\item The \emph{core} of $B$ is the divisor remaining after removing all rational twigs of $B$.

\item The \emph{core graph} of $B$ is the graph resulting from the contraction of vertices of the dual graph of $D$ which have degree $2$ and correspond to components contained in some rational twig of the graph.

\item Assume $B$ is a tree. The (non-weighted) \emph{Eisenbud-Neumann diagram} of $B$ is the graph resulting from contraction of vertices of the dual graph of $B$ of degree $2$ which correspond to rational components.

\eenum\edfn

The core graph contains the dual graph of the core and remembers to which components of the core maximal rational twigs were attached (vertices of degree $1$ are not contracted). Note that even if $B$ is an exceptional divisor of a minimal log resolution of a cusp (and hence is a three-valent tree of special type), the core graph is in general bigger than the Eisenbud-Neumann diagram.

\subsection{Rational cuspidal curves}\label{ssec:cuspidal_curves}

For a rational curve $\bar E\subseteq \PP^2$ we put $\kappa_{1/2}(\PP^2,\bar E)=\kappa(K_X+\frac{1}{2}D)$ and $p_2(\PP^2,\bar E)=h^0(2K+D)$, where $(X,D)\to (\PP^2,\bar E)$ is the minimal log resolution.

\blem\label{lem:cuspidal_of_gt_and_khalf} Let $\bar E\subseteq \PP^2$ be a rational cuspidal curve and let $\pi\:(X,D) \to (\PP^2,\bar E)$ be any weak resolution of singularities. Then: \benum[(i)]

\item $\PP^2\setminus\bar E$ is $\Q$-acyclic.

\item The numbers $h^0(m(2K_X+D))$, and hence $\kappa_{1/2}(\PP^2,\bar E)$ and $h^0(\PP^2,\bar E)$, do not depend on the choice of $\pi$.

\item If $\kappa_{1/2}(\PP^2,\bar E)\geq 0$ then for a fiber $f$ of any $\PP^1$-fibration of $X$ we have $f\cdot D\geq 4$.

\item If $\PP^2\setminus \bar E$ is of log general type then it contains no topologically contractible curves (i.e.\ curves which are homotopic to a point). In case $\pi$ is the minimal log resolution we have $(K+D)^-=\Bk D$.

\eenum \elem

\begin{proof} (i) The $\Q$-acyclicity follows from the Lefschetz duality.

(ii)  Let $\sigma\:X'\to X$ be a blowup and let $D'=\sigma_*^{-1}D+\Exc \sigma$. Let $\mu$ be the number of components of $D$ passing through the center of $\sigma$. Denote the proper transform of $\bar E$ on $X$ by $E$. The divisor $D-E$ is snc and $E$ is smooth, so $\mu\leq 3$. Clearly, $\sigma_*$ embeds the linear system of $m(2K_{X'}+D')$ into the linear system of $$\sigma_*(m(2K_{X'}+D'))=m(2K_X+D).$$ Now if $m(2K_X+D)\sim U$ then $$m(2K_{X'}+D')\sim \sigma^*U+m(3-\mu)\Exc \sigma\geq \sigma^*U,$$ so $\sigma_*$ is surjective.

(iii) Let $f$ be a smooth fiber of a $\PP^1$-fibration of $X$. By (ii) $\kappa(2K_X+D)\geq 0$, so $0\leq f\cdot (2K_X+D)=-4+f\cdot D$.

(iv) By \cite{MiTs-lines_on_qhp} $\PP^2\setminus\bar E$ contains no topologically contractible curves. If $\pi$ is the minimal log resolution then by \ref{lem:Bk}(iv) $(K_X+D)^-=\Bk D$.
\end{proof}

There are many conjectures related to cuspidal curves (see \cite{FLMN}), we mention only those related to our results. Recall that $\bar E\subseteq \PP^2$ is \emph{projectively rigid} if every embedded equisingular deformation of it is projectively equivalent to $\bar E$.

\begin{conjecture}\label{con:conjectures} Let $\bar E\subseteq \PP^2$ be a rational cuspidal curve and let $\pi\:(X,D)\to (\PP^2,\bar E)$ be the minimal log resolution of singularities. Denote the proper transform of $\bar E$ on $X$ by $E$ and the logarithmic tangent sheaf of $X$ along $D$ by $\cal T_X(-\log D)$. We have the following conjectures: \benum[(1)]
\item (The rigidity conjecture, \cite{FZ-deformations}, \cite{FLZa-cusp}). If $\bar E\subseteq \PP^2$ is of log general type then it is projectively rigid and has unobstructed deformations. Equivalently, $h^i(\cal T_X(-\log D))=0$ for $i=1,2$.
\item (The weak rigidity conjecture, \cite{FLZa-cusp}). If $\bar E\subseteq \PP^2$ is of log general type then the Euler characteristic of $\cal T_X(-\log D)$ vanishes. Equivalently, $K_X\cdot(K_X+D)=0$. Equivalently (see \ref{lem:properties_of_(Xi,Di)}(i)), $p_2(\PP^2,\bar E)=0$.
\item (The Coolidge-Nagata conjecture, \cite{Coolidge}, \cite{Nagata}). There exists a Cremona transformation of $\PP^2$ which maps $\bar E$ onto a line. Equivalently (\cite[2.6, 3.4]{Kumar-Murthy}, \cite{Coolidge}), $h^0(2K_X+E)=0$.
\item (Finiteness conjecture, \cite[1.6, p.16]{Zaid-open_MONTREAL_problems},\cite[\S 15]{open_problems_AAG}). The set of possible Eisenbud-Neumann diagrams of $D$ is finite.
\eenum\end{conjecture}

Note that the equivalent formulations of conjectures (1) and (2) show that the weak rigidity conjecture implies the Coolidge-Nagata conjecture when $\PP^2\setminus\bar E$ is of log general type. In \ref{con:k_half_vanishes} we formulate an even stronger conjecture, which is more natural in our setting. The mentioned equivalent formulations show also that our choice of the setup in which we run the log MMP (the $\frac{1}{2}$ in $K_X+\frac{1}{2}D$) is the most natural one.

\bprop\label{prop:Fujita_boundary_excluded} If $\PP^2\setminus \bar E$ is not of log general type then it is $\C^1$- or $\C^*$-fibered and $\kappa_{1/2}(\PP^2,\bar E)=-\8$. In particular, $p_2(\PP^2,\bar E)=0$ and $\bar E$ satisfies the Coolidge-Nagata conjecture. \eprop

\begin{proof} Assume $S=\PP^2\setminus\bar E$ is not of long general type. Suppose it is neither $\C^1$- nor $\C^*$-ruled. By structure theorems for smooth affine surfaces (see \cite[3.1.3.1, 3.1.7.1]{Miyan-OpenSurf}) $\kappa(S)=0$. Let $(X,D)$ be the minimal log resolution of $(\PP^2,\bar E)$. Since $S$ is not $\C^*$-fibered, by \cite[3.4.4.3]{Miyan-OpenSurf} $(X,D)$ is almost minimal. By \cite[8.64]{Fujita-noncomplete_surfaces} (cf. \cite[3.4.4.2]{Miyan-OpenSurf}) $\PP^2\setminus \bar E$ is one of the three Fujita surfaces $Y\{a,b,c\}$, so $D$ is a fork whose maximal twigs are $(-2)$-chains. But then $c=1$ and there should exist a tip of $D$ (namely $E$) for which $D-E$ is negative definite. This is not so for the surfaces $Y\{a,b,c\}$, because the branching components have self-intersections $1,0,-1$ respectively; a contradiction.

Thus $S$ is $\C^1$- or $\C^*$-fibred. The fibration extends to a $\PP^1$-fibration of some weak resolution $(X,D)\to (\PP^2,\bar E)$, so that $f\cdot D\leq 2$. But then $\kappa_{1/2}(\PP^2,\bar E)=\kappa(K_X+\frac{1}{2}D)=-\8$, because $f\cdot(K_X+\frac{1}{2}D)<0$. It follows that $p_2(\PP^2,\bar E)=h^0(2K_X+E)=0$, so we are done due to \ref{lem:cuspidal_of_gt_and_khalf}(iii).
\end{proof}

\section{Minimal models: construction}\label{sec:logMMP}

From now on we assume that $\bar E\subseteq \PP^2$ is a (closed) singular cuspidal rational curve in $\PP^2$ and that $\pi_0\:(X_0,D_0)\to (\PP^2,\bar E)$ is the minimal weak (see \ref{def:resolution}) resolution of singularities. Let $\pi\:(X,D)\to (\PP^2,\bar E)$ be the minimal log resolution. Clearly, we have a birational factorizing morphism $\psi_0\:(X,D)\to (X_0,D_0)$. Let $E_0$ and $E$ be the proper transforms of $\bar E$ on $X_0$ and $X$ respectively. If $\PP^2\setminus\bar E$ is not of log general type then by \ref{prop:Fujita_boundary_excluded} it is $\C^1$- or $\C^*$-fibered, $\kappa_{1/2}(\PP^2,\bar E)=-\8$ and $\bar E$ satisfies the Coolidge-Nagata conjecture. Cuspidal curves of this type are classified (see \cite[5.2]{FLMN} and references there). In fact one can derive the classification from what is known on $\C^1$- and $\C^*$-fibered $\Q$-homology planes (see \cite{MiSu-Qhp_quotient_singularities}). Therefore, we may and shall assume that $\PP^2\setminus \bar E$ is of log general type.

\brem[The strategy]\label{rem:strategy} Our basic idea is to run the logarithmic minimal model program for $(X,\frac{1}{2}D)$ (or rather for $(X_0,\frac{1}{2}D_0)$, but this is just a technical detail). At this point we should remark that one can of course run the log MMP in the usual way, i.e.\ for $(X,D)$. But because $X\setminus D$ is a smooth affine surface of log general type which has Euler characteristic $1$, it cannot contain lines (see \cite[3.4.10.1]{Miyan-OpenSurf}), so only the rays supported in $D$ get contracted. Unfortunately, although the process of minimalization in this case does not contract any curve in the affine part, the nefness of the log canonical divisor of the minimal model gives rather weak consequences. On the other hand, as we will see, running the log MMP for $(X,\frac{1}{2}D)$ leads to contracting curves in the affine part, but now the nefness of the log canonical divisor or the log Mori space structure of the minimal model give strong bounds. Crucially, if the open part of the surface is not $\C^{**}$-fibered then $K_Y+\frac{1}{2}D_Y$ is nef or anti-ample, so its square is non-negative (see \ref{thm:MAIN1detailed}(6) for the resulting inequality). \erem

Since $(\Pic \PP^2)\otimes \Q\cong \Q$ is generated by $\bar E$, $(\Pic X)\otimes \Q$ is generated freely by the components of $D$. Because $D-E$ has a negative definite intersection matrix, the intersection matrix of $D$ is not negative definite by the Hodge index theorem. Minimality of the log resolution $\pi\:(X,D)\to (\PP^2,\bar E)$ implies that all $(-1)$-curves of $D$ are branching. Let $Q_i$ be the reduced exceptional divisor over the cusp $q_i$. It is a rational tree with a negative definite intersection matrix. There is a unique $(-1)$-curve in $Q_i$, say $L_i$, and it is contained in a maximal twig of $Q_i$. We have $E\cdot Q_i=E\cdot L_i=1$. Because $\pi$ is minimal, $L_i$ is not a tip of $D$, so $Q_i-L_i$ has two connected components. One of them is a rational chain and the other a rational tree; $L_i$ meets them in tips. Because $Q_i$ contracts to a smooth point, it is in fact a three-valent tree with $d(Q_i)=d([1])=1$. Furthermore, $Q_i$ can be seen as being produced by a \emph{connected sequence of blow-ups}, i.e.\ we can decompose the morphism contracting it to a point into a sequence of blow-ups $\sigma_1\circ\ldots\circ\sigma_s$, so that then the center of $\sigma_{i+1}$ belongs to the exceptional component of $\sigma_i$ for $i\geq 1$. The Eisenbud-Neumann diagram of $Q_j$ is:

$$ 
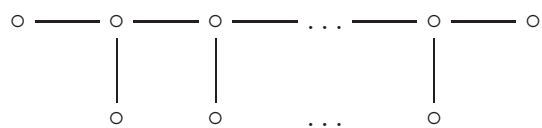
\xymatrix{ {\circ}\ar@{-}[r] &{\circ}\ar@{-}[r]\ar@{-}[d] &{\circ}\ar@{-}[r]\ar@{-}[d] &{\ldots}\ar@{-}[r] &{\circ}\ar@{-}[r]\ar@{-}[d] &{\circ} \\ {} &{\circ} &{\circ} &{\ldots} &{\circ} & {} }$$

\vskip 0.3cm

Now we analyze the logarithmic Minimal Model Program run for the log surface $(X_0,\frac{1}{2}D_0)$. For basic notions and theorems of the program the reader is referred to \cite{KollarKovacs-2DlogMMP} and \cite{Matsuki, KollarMori}. Recall that if $Y$ is a normal projective surface then a morphism $\alpha\:Y\to Z$ onto a normal variety $Z$ is called a \emph{contraction} if it has connected fibers. We define $\rho(Y)=\rk \NS(Y)$, where $\NS(Y)$ is the Neron-Severi group of $Y$, and we put $\rho(\alpha)=\rho(Y/Z)=\rho(Y)-\rho(Z)$. If $D$ is an effective $\Q$-divisor on $Y$ then a contraction $\alpha\:Y\to Z$ is \emph{extremal} with respect to $K_Y+D$ if and only if $\rho(\alpha)=1$. We will use only contractions for which the contracted numerical class intersects $K_Y+D$ negatively.

\bdfn\ \benum[(i)]

\item A cusp of a planar curve is \emph{semi-ordinary} if it is locally analytically isomorphic to the singular point of $x^2=y^{2m+1}$ at $0\in\Spec \C[x,y]$ for some $m\geq 1$.

\item Let $T$ be a reduced effective divisor. A subdivisor of $T$ of type $C=[2,1,3,(2)_{m-1}]$ with $m\geq 1$, such that $T-C$ meets $C$ exactly in the $(-1)$-curve of $C$, in one point nad transversally, is a \emph{semi-ordinary ending of $T$}.
\eenum\edfn

Note that an \emph{ordinary (simple) cusp} is exactly a semi-ordinary cusp for which $m=1$. The exceptional divisor of the minimal resolution of a semi-ordinary cusp is $[2,1,3,(2)_{m-1}]$, where $(2)_{m-1}$ is a chain of $(-2)$-curves of length $m-1\geq 0$. For instance, $D$ contains a semi-ordinary ending if and only if $\bar E$ has at least one semi-ordinary cusp. It is easy to see that all components of a semi-ordinary ending of $D$ are contracted when the logarithmic Minimal Model Program is run for the log surface $(X,\frac{1}{2}D)$. On the other hand, since $E_0$ is tangent to all $(-1)$-curves of $D_0$, $D_0$ contains no semi-ordinary endings.

We now define by induction a finite sequence of birational contractions $$\psi_i\:(X_{i-1},D_{i-1})\to (X_i,D_i),\  1\leq i\leq n$$ starting from the minimal weak resolution $(X_0,D_0)$ defined above.

\bnot\label{def:Delta,Upsilon,Dflat} Assume $(X_i,D_i)$ is defined. Write $K_i$ for the canonical divisor on $X_i$. \benum[(i)]

\item Let $\Delta_i$ be the sum of all maximal $(-2)$-twigs of $D_i$.

\item Let $\Upsilon_i$ be the sum of $(-1)$-curves $L$ in $D_i$, for which either $\beta_{D_i}(L)=3$ and $L\cdot \Delta_i=1$ or $\beta_{D_i}(L)=2$ and $L$ meets exactly one component of $D_i$.

\item Decompose $\Delta_i$ as $\Delta_i=\Delta^+_i+\Delta^-_i$, where $\Delta^+_i$ consists of these $(-2)$-twigs of $D_i$ which meet $\Upsilon_i$.

\item Put $D_i^\flat=D_i-\Upsilon_i-\Delta^+_i-\Bk'\Delta^-_i$, where $\Bk' \Delta_i^-=\Bk_{D_i}(\Delta_i^-)$ (see Section \ref{sec:preliminaries}). \eenum\enot

We now denote the properties of $\psi_i$ and the pairs $(X_i,D_i)$ we will prove: \benum[\ \ ($P_1$)]

    \item $X_i$ is smooth,

    \item $D_i$ contains no semi-ordinary endings or superfluous $(-1)$-curves,

    \item $D_i-E_i$, where $E_i=\psi_{i*}E_{i-1}$, is an snc-divisor,

    \item $E_i$ is a smooth rational curve,

    \item $D_i$ is not an snc-divisor,

    \item $D_i-\Delta_i-\Upsilon_i$ is connected,

    \item Components of $\Upsilon_i$ are disjoint, each meets at most one component of $\Delta_i$.

    \item $X_i\setminus D_i$ is affine and contains no affine lines.
\eenum

\noin Note that (when proved) property $(P_7)$ implies that $b_0(\Delta_i^+)\leq \#\Upsilon_i$ and that $\Upsilon_i+\Delta_i$ can be contracted by a birational morphism.

\blem\label{lem:producing_A} Assume $(X_i,D_i)$ satisfies $(P_1)$-$(P_8)$. Let $\alpha_i\: (X_i,D_i)\to (Y_i,D_{Y_i})$ be the contraction of $\Upsilon_i+\Delta_i$. Then \benum[(i)]

    \item $\alpha_i^*(K_{Y_i}+\frac{1}{2}D_{Y_i})=K_i+\frac{1}{2}D_i^\flat$.

    \item $(Y_i,\frac{1}{2}D_{Y_i})$ has log terminal singularities (with discrepancies $\geq -\frac{1}{2}$).

    \item If $\theta\:{Y_i}\to Z$ is a birational extremal contraction, negative with respect to $K_{Y_i}+\frac{1}{2}D_{Y_i}$, then the contracted curve is not a component of $D_{Y_i}$.

\eenum \elem

\begin{proof} (i) Since $\Upsilon_i+\Delta_i$ is contracted by $\alpha_i$, it has a negative definite intersection matrix, so the divisor $\alpha^*(K_{Y_i}+\frac{1}{2}D_{Y_i})-K_i$ is determined uniquely by the fact that, by the projection formula, it intersects trivially with every component of $\Upsilon_i+\Delta_i$. But it follows from the definition of $D_i^\flat$ that this is the case for $\frac{1}{2}D_i^\flat$.

(ii) From (i) we see that if we write $$K_i+\alpha_{i*}^{-1}(\frac{1}{2}D_{Y_i})=\alpha_i^*(K_{Y_i}+\frac{1}{2}D_{Y_i})+\sum a_UU,$$ where the sum runs over irreducible curves $U$ contracted by $\alpha$, then the discrepancies $a_U$ are bigger than $-\frac{1}{2}$. It remains to show that singularities of $(Y_i,D_{Y_i})$ coming from tangency of some components of $D_0$ are log terminal. Let $R_1+U_1+U_2+\ldots+U_n+R_2$ be a chain of smooth rational curves on a smooth surface $V$, such that $U=U_1+U_2+\ldots+U_n$ is of type $[1,2,\ldots,2]$ and let $R_3$ be a smooth rational curve intersecting the chain in $U_1$, transversally. Put $R=R_1+R_2+R_3$. It is enough to check that for $R+U$ the contraction of $U$ creates at most log terminal singularities. We allow $R_2=0$, which corresponds to $s_j=1$ for some $j$. If $p\:V\to Y$ denotes the contraction then we get $$p^*p_*(K_V+\frac{1}{2}(R+U))=K_V+\frac{1}{2}R$$ in case $R_2=0$ and $$p^*p_*(K_V+\frac{1}{2}(R+U))=K_V+\frac{1}{2}R+U$$ in case $R_2\neq 0$, so the discrepancies are equal to $0$ and $-\frac{1}{2}$ respectively.

(iii) Suppose $B'=\Exc \theta$ is a component of $D_{Y_i}$. Let $B$ be the proper transform of $B'$ on $X_i$. Put $R=D_i-\Delta_i-\Upsilon_i$. Using the projection and adjunction formulas we get $$0>(2K_{Y_i}+D_{Y_i})\cdot B'=(2K_i+D_i^\flat)\cdot B=\beta_{R}(B)+\Delta^-_i\cdot B -4-B^2-B\cdot\Bk'\Delta^-_i.$$ By \ref{lem:Bk}(ii) if $V$ is the component of $D_i$ meeting a given twig $T$ then $V\cdot \Bk'T=\frac{1}{d(T)}\leq \frac{1}{2}$, hence $$0\leq 2 \beta_R(B)+B\cdot\Delta^-_i\leq 7+2B^2.$$ We get $B^2\geq -3$.  Since $\alpha_i(B)=B'\neq 0$, the definition of $\Delta_i^+$ shows that $B\cdot \Delta_i^+=0$.

Suppose $B^2=-3$. Then $\beta_R(B)=0$ and $B\cdot\Delta^-\leq 1$. Negative definiteness of the intersection matrix of $B+\Upsilon_i+\Delta_i$ implies that $B\cdot \Upsilon_i\leq 1$. Moreover, if $L\in \Upsilon_i$ meets $B$ and $\Delta_L$ is the connected component of $\Delta_i^+$ meeting $L$ then $\Delta_L=[2]$, so $B$ is a part of a semi-ordinary ending of $D_i$, which contradicts the assumption. It follows that $B^2=-2$ or $B^2=-1$. In both cases the definition of $\Delta_i$ together with the negative definiteness of $B+\Upsilon_i+\Delta_i$ imply that $B\cdot (\Upsilon_i+\Delta_i^+)=0$ and $B\cdot \Delta_i^-\leq 1$. We obtain $$0\leq 2 \beta_D(B)\leq B\cdot\Delta^-_i+7+2B^2.$$

Suppose $B\cdot \Delta_i^-=1$. For $B^2=-2$ we get $\beta_D(B)\leq 2$, which contradicts the definition of $\Delta_i^-$. For $B^2=-1$ we get $\beta_D(B)\leq 3$, so $B\in \Upsilon_i$; a contradiction. Therefore $B\cdot \Delta_i^-=0$ and $\beta_D(B)\leq 3+B^2$. If $B^2=-1$ then $B$ is a superfluous $(-1)$-curve and if $B^2=-2$ then it is a $(-2)$-tip of $D$. In both cases we arrive at a contradiction.
\end{proof}

\bcor\label{cor:A_i_exists} If $K_i+\frac{1}{2}D_i^\flat$ is not nef and $\alpha_i(X_i,\frac{1}{2}D_i)$ is not a log Mori fiber space then there exists a $(-1)$-curve $A_i$ on $X_i$, not a subdivisor of $D_i$, which is $(K_i+\frac{1}{2}D_i^\flat)$-negative. Moreover, $$A_i\cdot(\Upsilon_i+\Delta_i^+)=0 \text{\ \ and\ \ } A_i\cdot (D_i-\Delta^-_i)=A_i\cdot\Delta^-_i=1$$ and the component of $\Delta^-_i$ meeting $A_i$ is a tip of $\Delta^-_i$.  \ecor

\begin{proof} By the contraction theorem of log MMP there exists an extremal birational contraction $Y_i\to Z$ of a $(K_{Y_i}+\frac{1}{2}D_{Y_i})$-negative curve. Let $A_i$ be the proper transform of the curve on $X_i$. By \ref{lem:producing_A} $A_i$ is not a subdivisor of $D_i$ and $(2K_i+D_i^\flat)\cdot A_i<0$. It follows that $A_i\cdot K_i<0$, so $A_i$ is a $(-1)$-curve and $A_i\cdot (D_i-\Upsilon_i-\Delta^+_i)<2+A_i\cdot \Bk'\Delta^-_i$. As we already noticed in the proof above, the negative definiteness of the intersection matrix of $B+\Upsilon_i+\Delta_i$ implies $A_i\cdot(\Upsilon_i+\Delta^+_i)=0$ and $A_i\cdot \Delta^-_i\leq 1$. It also implies that the component of $\Delta^-_i$ meeting $A_i$ is a tip of $\Delta_i$ ($\Delta_i$ may have two tips, and it does not have to be the one which is a tip of $D_i$). By $(P_8)$ we have $A_i\cdot D_i\geq 2$, so $2\leq A_i\cdot D_i<2+A_i\cdot \Bk'\Delta^-_i\leq 3$. Then $A_i\cdot \Delta^-_i=1$ and $A_i\cdot D_i=2$.
\end{proof}

We now define $(X_i,D_i)$ by induction. The definition will be completed once we check that properties  $(P_1)$-$(P_8)$ indeed hold for $(X_{i_0+1},D_{i_0+1})$. For this see \ref{prop:(Xi,Di)_properties}.

\bdfn\label{def:psi_and_models} Assume $(X_i,D_i)$ and $\psi_i$ for $0\leq i\leq i_0$ are defined and that the pair $(X_i,D_i)$ satisfies properties $(P_1)$-$(P_8)$. Let $\alpha_i\:(X_i,D_i)\to (Y_i,D_{Y_i})$ be as in \ref{lem:producing_A}. \benum[(i)]

\item If $K_{i_0}+\frac{1}{2}D_{i_0}^\flat$ is nef or if $\alpha_{i_0}(X_{i_0},\frac{1}{2}D_{i_0})$ is a log Mori fiber space then we put $n=i_0$ and so $\psi_{i_0}$ is the last morphism in the sequence. Otherwise we define $$\psi_{i_0+1}\:(X_{i_0},D_{i_0})\to (X_{i_0+1},D_{i_0+1}),$$ where $D_{i_0+1}=\psi_{(i_0+1)*}D_i$, as the composition of the contraction of some $A_{i_0}$ given by \ref{cor:A_i_exists} and successive contractions of superfluous $(-1)$-curves in the image of $D_{i_0}$.

\item We say that $\psi_{i+1}$ is a \emph{contraction of type $II$} if it contracts both components of $D_i$ meeting $A_i$. Otherwise it is of \emph{type $I$}.

\item We call $(Y_n,\frac{1}{2}D_{Y_n})$ a \emph{minimal model} of $(X_0,\frac{1}{2}D_0)$.

\item Let $\varphi_n\:(X_n',D_n')\to (Y_n,D_n)$ be the minimal log resolution. We call $(X_n',\frac{1}{2}D_n')$ an \emph{almost minimal model} of $(X_0,\frac{1}{2}D_0)$.
\eenum\edfn

\bsrem The definition of an almost minimal model is analogous to the definition in the theory of log surfaces in case of a simple normal crossing boundary with integral coefficients (see \cite[2.3.11, 2.4.3]{Miyan-OpenSurf}\footnote{In loc.\ cit.\ the minimal model is called 'relatively minimal'. This however is in conflict with modern terminology.}). We want to emphasize that at each step of the construction there may be more than one choice of $A_i$ (equivalently of $\psi_{i+1}$) and a priori the final model $(X_n,D_n)$ and in fact also the number '$n$' of the steps taken, may depend on these choices. This will not cause any problems. We simply work with a (any) fixed choice of a sequence of $A_i$'s.

Not also that in our approach we study a minimal model of $(X_0,D_0)$ and because of this $E_0=\psi_0(E)$ is tangent to $D_0-E_0$ and is never contracted by $\psi_n\circ\ldots\circ\psi_1$. It is essentially a matter of choice, but in a general situation, when the analysis is repeated for a general pair $(X,D)$ it is more natural to define $\psi_0\:(X,D)\to (X_0,D_0)$ as a birational morphism with smallest possible $\rho(\psi_0)$ which contracts (some) components of $D$, and such that $D_0$ contains no $(K_0+\frac{1}{2}D_0)$-negative components of negative self-intersection. However, doing so in our situation we see that this new version of $\psi_0$ contracts fewer curves and it may happen, for example for $c\leq 2$, $E_n^2=-1$, that the curve $E_n$ itself is $(K_n+\frac{1}{2}D_n)$-negative (and $D_n$ contains superfluous $(-1)$-curves), which would cause some additional problems.
\esrem

\blem\label{lem:amm_is_snc_min} Let $(X_n',\frac{1}{2}D_n')$ be an almost minimal model of the minimal weak resolution of singularities of $(\PP^2,\bar E)$ as above. If $D_n'$ is not snc-minimal then there exists a $\PP^1$-fibration of $X_n'$ such the intersection index of $D_n'$ with a general fiber is at most three. In particular, $\kappa_{1/2}(\PP^2,\bar E)=-\8$.  \elem

\begin{proof} Suppose $V$ is a superfluous $(-1)$-curve in $D_n'$. Assume first that $V$ is a proper transform of $E_n$. By the definition of $X_n'$ the component $W$ of $D_n'-V$ over a given cusp of $\bar E$ which meets $V$ is a $(-1)$-curve, non-branching in $D_n'-V$. Because $V$ is superfluous, for $f=V+W$ we have $f\cdot D_n\leq\beta_{D_n'}(V)+1\leq 3$, so we are done. Assume $V$ is not a proper transform of $E_n$. Since there are no superfluous $(-1)$-curves in $D_n$, some component $L$ of $\Exc \varphi_n$ meets $V$. The morphism $\varphi_n\:X_n'\to X_n$ is a log resolution of non-snc points of $D_n$ and all these points belong to $E_n$. It follows from the definition of $\psi_i$, $i\geq 0$ that for each such point at most two components of $D_n-E_n$ pass through it, exactly one of them is tangent to $E_n$, and each component of $D_n-E_n$ passes through at most one of such point. We may therefore assume that there is only one non-snc point in $D_n$. The curve $V'=\varphi_n(V)$ meets $E_n$, because $\varphi_n$ is minimal. If it is tangent to $E_0$ then $L^2=-1$ and $L$ meets at most two components of $D_n'$ other than $V$. Because $V$ is superfluous, for $f=L+V$ we have $f\cdot D\leq 3$ and $|L+V|$ gives the desired fibration. Assume that $V'$ meets $E_n$ transversally. Then $\varphi_n$ touches $V$ once, so $V'$ is a $0$-curve for which $$V'\cdot D_n=\beta_{D_n}(V')=\beta_{D_n'}(V)+1\leq 3$$ and hence $|f+V|$ gives the fibration. By \ref{lem:cuspidal_of_gt_and_khalf}(iii) $\kappa_{1/2}(\PP^2,\bar E)=-\8$.
\end{proof}

\brem\label{rem:diagram}  It follows from the definition that $\psi_0\:(X,D)\to (X_0,D_0)$ is a minimal log resolution of non-snc points of $D_0$ (in fact the points of tangency, all belonging to $E_0$), because the components of $D_0$ are smooth. For $i=1,\ldots,n$ let $\varphi_i\:(X_i',D_i')\to (X_i,D_i)$ be the minimal log resolution (which again is a resolution of non-snc points). Put $(X_0',D_0')=(X,D)$. Because $\Exc \psi_i$ and $\Exc \varphi_i$ are disjoint we have a lifting $\psi'_i\:(X_{i-1}',D_{i-1}')$ of $\psi_i$. Because $\alpha_i\circ\psi_i$ contracts $\Exc \alpha_{i-1}$ we have a morphism $\psi_i''\:(Y_{i-1},D_{Y_{i-1}})\to (Y_i,D_i)$ which makes the following diagram commute:

$$ \xymatrix{
{} & {(X,D)}\ar@{>}[r]^-{\psi_1'}\ar@{>}[d]^{\psi_0}\ar@{>}[ld]_-{\ \pi} &{(X_1',D_1')}\ar@{>}[r]^-{\psi_2'\ }\ar@{>}[d]^{\varphi_1} &{\ldots}\ar@{>}[r]^-{\psi_n'\ } &{(X_n',D_n')}\ar@{>}[d]^{\varphi_n}\\
{(\PP^2,\bar E)} & {(X_0,D_0)}\ar@{>}[r]^-{\psi_1}\ar@{>}[d]^{\alpha_0}\ar@{>}[l]^-{\ \pi_0} &{(X_1,D_1)}\ar@{>}[r]^-{\psi_2\ }\ar@{>}[d]^{\alpha_1} &{\ldots}\ar@{>}[r]^-{\psi_n\ } &{(X_n,D_n)}\ar@{>}[d]^{\alpha_n}
\\ {} & {(Y_0,D_{Y_0})}\ar@{>}[r]^-{\psi_1''} & {(Y_1,D_{Y_1})}\ar@{>}[r]^-{\psi_2''\ } & {\ldots}\ar@{>}[r]^-{\psi_n''\ } & {(Y_n,D_{Y_n})}}$$

 \erem

\section{Minimal models: properties}

We now analyze the process of minimalization in detail. We keep the notation from the previous section. Recall that $A_i\subset X_i$ meets a tip of $\Delta_i$, but this tip does not have to be a tip of $D_i$. Put $\psi=\psi_n\circ\ldots\circ\psi_1$ and $\psi'=\psi_n'\circ\ldots\circ\psi_1'$.

\bprop\label{prop:(Xi,Di)_properties} Let $(X_{i+1},D_{i+1})$ and $\psi_{i+1}$, $i=0,\ldots,n-1$ be as above. Then: \benum[(i)]

\item $X_{i+1}$, $D_{i+1}$ and $E_{i+1}$ have the properties $(P_1)$--$(P_8)$ listed after \ref{def:Delta,Upsilon,Dflat}.

\item $\psi_{i+1}^*(K_{i+1}+D_{i+1})=K_i+D_i+A_i$.

\item $\psi_{(i+1)*}(\Upsilon_i)\leq\Upsilon_{i+1}$.

\item If $A_i$ meets $\Delta_i^-$ not in a tip of $D_i$ then $\psi_{i+1}$ contracts only $A_i$ (hence is of type $I$) and we have $b_0(\Delta_{i+1})=b_0(\Delta_i)$ and  $b_0(\Delta^+_{i+1})-b_0(\Delta^+_i)=\#\Upsilon_{i+1}-\#\Upsilon_i=1$.

\item If $A_i$ meets $\Delta_i^-$ in a tip of $D_i$ then $\psi_{i+1}$ contracts at least $A_i$ and the connected component of $\Delta_i^-$ meeting $A_i$ (with a minor exception when this connected component and $A_i$ meet the same irreducible component of $D_i$). We have $b_0(\Delta_{i+1})=b_0(\Delta_i)-1$ and $b_0(\Delta^+_{i+1})-b_0(\Delta^+_i)\leq \#\Upsilon_{i+1}-\#\Upsilon_i\in \{0,1\}$.

\item For every component $U$ of $D_i-\Delta_i$ we have $U\cdot \Delta_i\leq 1$.

\item For each component $U$ of $D_0-E_0$ we have $\psi(U)\cdot E_n\leq U\cdot E_0+1$. If the equality holds then the unique $\psi_i$ increasing the intersection of the images of $U$ and $E_0$ is of type $I$. Moreover, it touches $U$ exactly once and either $U$ is the component of $\Delta_0^-$ met by $A_{i-1}$ or there is a unique connected component of $\Delta_0^-$ meeting $U$ and this component is contracted by $\psi_i$.

\item For every $m\geq 1$ we have $h^0(m(2K_i+D_i))=h^0(m(2K+D)).$
\eenum\eprop

\begin{proof} We proceed by induction on $i$.

(i) We only need to check $(P_2)$, $(P_6)$ and $(P_7)$, because the remaining properties follow from the definition of $\psi_{i+1}$. First of all, $D_0$ does not contain semi-ordinary endings by the definition of $\psi_0$. Suppose $T$ is a semi-ordinary ending of $D_{i+1}$. We may assume that the point $\psi_{i+1}(A_i)$ belongs to $T$, otherwise the proper transform of $T$ is a semi-ordinary ending of $D_i$. But $T\cdot (D_{i+1}-T)=1$, so since $A_i$ intersects $D_i$ in exactly two points, it follows that $D_i-A_i$ is not connected; a contradiction. Properties $(P_6)$ and $(P_7)$, which clearly hold for $(X_0,D_0)$, will follow by induction from the proof of the remaining parts of the proposition. $(P_8)$ holds for $X_0\setminus D_0$ by \ref{lem:cuspidal_of_gt_and_khalf}(iv) and holds in general because the complement of $X_{i+1}\setminus D_{i+1}$ in $X_i\setminus D_i$ is a closed subset isomorphic to $\C^*$.

(ii) is a consequence of the fact that $\psi_{i+1}$ is inner for $D_i+A_i$.

(iii) Suppose $L$ is a component of $\Upsilon_i$, such that $\psi_{i+1}(L)$ is not a component of $\Upsilon_{i+1}$. Since $A_i\cdot (\Upsilon_i+\Delta_i^+)=0$, there is a component $B$ in $D_i-\Upsilon_i-\Delta_i^+$ meeting $L$ for which $(\psi_{i+1})_*(B)=0$. It follows that $B\cdot L=1$. Since $\psi_{i+1}$ is inner for $D_i+A_i$, $B$ is contained in a maximal twig of $D_i$. But then (P6) for $(X_i,D_i)$ fails; a contradiction.

(vi) Suppose $U\cdot \Delta_i\geq 2$ and let $U_0$ be the proper transform of $U$ on $X_0$. Then $U_0$ meets two $(-2)$-twigs of $D_0$. Clearly, $U_0$ and these twigs are contained in $D_0-E_0$. Think of the morphism $\pi_0\:X_0\to \PP^2$ as a composition of blowdowns inside $D_0$. At some point $U_0$ becomes a $(-1)$-curve and the $(-2)$-twigs are not yet touched. Thus some image of $D_0-E_0$ contains the chain $[2,1,2]$, so the intersection matrix of $D_0-E_0$ is not negative definite. This is a contradiction.

\ssk

Let $V$ and $W$ be the components of $D_i$ meeting $A_i$ which are subdivisors of $\Delta^-_i$ and $D_i-\Delta_i-\Upsilon_i$ respectively. Let $\Delta_V$ be the connected component of $\Delta^-_i$ containing $V$.

(iv) Assume $V$ is not a tip of $D_i$. The contraction of $A_i$ does not touch components other than $V$ and $W$. It follows from the assumption that the image of $V$ after this contraction is a branching $(-1)$-curve, a component of $\Upsilon_{i+1}$. Even if it happens that the image of $W$ is a $(-1)$-curve, then it is also branching, because $W$ is not a $(-2)$-tip of $D_i$. Moreover $W\nleq\Delta_i$, so $\psi_{i+1}(W)\nleq\Upsilon_{i+1}$. Therefore, $\psi_{i+1}$ contracts only $A_i$, $\psi_{(i+1)*}(\Delta_V-V)\leq \Delta^+_{i+1}$ and $\psi_{i+1}(V)$ is the only component of $\Upsilon_{i+1}$ which is not an image of a component of $\Upsilon_i$. Thus $b_0(\Delta^+)$ increases by $1$ under $\psi_{i+1}$.  By (iii) images of components of $\Upsilon_i$ are components of $\Upsilon_{i+1}$, so we are done. Clearly, $(P_6)$ and $(P_7)$ for $(X_{i+1},D_{i+1})$ hold.

(v) Assume $V$ is a tip of $D_i$. Then $b_0(\Delta_{i+1})=b_0(\Delta_i)-1$. If $\Delta_V$ and $A_i$ meet a common component of $D_i$ then $\psi_{i+1}$ contracts $A_i$ and all but one irreducible components of $\Delta_V$. In this case $b_0(\Delta_{i+1})-b_0(\Delta_i)=-1$, $b_0(\Delta^+_{i+1})-b_0(\Delta^+_i)=0$ and $\#\Upsilon_{i+1}-\#\Upsilon_i=1$. We may therefore assume $\Delta_V$ and $A_i$ do not meet a common component of $D_i$. Then $\psi_{i+1}$ contracts at least $A_i+\Delta_V$.

Assume $W$ is not a tip of $D_i$. Then $\psi_{i+1}$ contracts only $A_i+\Delta_V$,  hence is of type $I$. If $T$ is the component of $D_i$ meeting $\Delta_V$ then $\psi_{i+1}(T)^2=T^2+1$, so $\psi_{i+1}(T)$ is not a component of $\Upsilon_{i+1}$, because either it is not a $(-1)$-curve or $\beta_{D_i}(T)\geq 3$ and by (vi) $T$ does not meet $\Delta_i-\Delta_V$. If it happens that $W$ meets $\Delta_i^-$, has $\beta_{D_i}(W)=2$ and $\psi_{i+1}(W)^2=-1$ then $\psi_{i+1}(W)$ is a component of $\Upsilon_{i+1}$ and the connected component of $\Delta_i^-$ meeting $W$, which is unique by (vi), becomes a connected component of $\Delta_{i+1}^+$. In other cases $\psi_{i+1}(W)$ is not a component of $\Upsilon_{i+1}$ and $b_0(\Delta^+)$ does not change. We obtain $b_0(\Delta^+_{i+1})-b_0(\Delta^+_i)= \#\Upsilon_{i+1}-\#\Upsilon_i\leq 1$. Again, we see that $(P_6)$ and $(P_7)$ hold.

Finally, assume $W$ is a tip of $D_i$. The contractions are inner for $D_i+A_i$ (in particular $E_i$ is not contracted) and they take place inside the sum of $A_i$ and the maximal twigs $T_V$ and $T_W$ of $D_i$ containing $V$ and $W$ respectively. Since $T_V$, $T_W$ are not contained in $\Delta_i^+$, by property $(P_6)$ for $X_i$ the components $V'$ and $W'$ are not components of $\Upsilon_i$. Let $V'$ and $W'$ be the branching components of $D_i$ meeting $T_V$ and $T_W$ respectively. Since each of $V', W'$ can meet $\Delta_i^-$ at most once, we have $b_0(\Delta_{i+1}^+)-b_0(\Delta_i^+)\in \{0,1,2\}$.

If $V'\neq W'$ then only $V'$ and $W'$ may become new components of $\Upsilon_{i+1}$ and we have $b_0(\Delta_{i+1}^+)-b_0(\Delta_i^+)= \#\Upsilon_{i+1}-\#\Upsilon_i$. On the other hand, if $V'=W'$ then $\beta_{D_i}(V')\geq 3$ and we see easily that $\psi_{i+1}(V')$ is not a component of $\Upsilon_{i+1}$, so $b_0(\Delta_{i+1}^+)-b_0(\Delta_i^+)=0$. However, in the latter case it may happen that $T_V+A_i+T_W$ contracts to a $(-1)$-curve, which then becomes a component of $\Upsilon_{i+1}$, hence $\#\Upsilon_{i+1}-\#\Upsilon_i\leq 1$. We see also that in any case, if $b_0(\Delta_{i+1}^+)-b_0(\Delta_i^+)\leq 1$ then $(P_6)$ and $(P_7)$ hold.

Suppose $b_0(\Delta_{i+1}^+)-b_0(\Delta_i^+)=2$. Then $W'\neq V'$. From the discussion above it follows also that $\psi_{i+1}$ contracts $T_V+A_i+T_W$ and makes $V'$ and $W'$ into $(-1)$-curves with $\beta_{D_{i+1}}=3$. In particular, at least one of $V'$, $W'$ is a $(-2)$-curve. Furthermore, there are $(-2)$-twigs $\Delta'_V$ and $\Delta'_W$ meeting $V'$ and $W'$ respectively (which become part of $\Delta^+$ after taking $\psi_{i+1}$). Since $\Delta'_V+V'+T_V$ and $T_W+W'+\Delta'_W$ are chains meeting the remaining part of $D_i$ once, they do not contain images of $A_j$'s (which are points) for $j<i$. Indeed, otherwise, because $A_j\cdot D_j=2$, reverting the contractions we would get that $D_0$ is disconnected. We therefore use the same letters for $\Delta'_V,V',T_V,T_W,W',\Delta'_W, A_i$ and their proper transforms on $X$. Since $F_\8=V'+T_V+A_i+T_W+W'$ contracts to $\psi_{i+1}(V'+W')$, and hence to a $0$-curve, there is a $\PP^1$-fibration $\lambda\:X\to \PP^1$ with $F_\8$ as one of the fibers. This fibration has exactly four sections contained in $D$, each intersecting a general fiber once. Since $E$ meets only $(-1)$-curves in $D$, it does not intersect $F_\8$, so it is vertical for $\lambda$. Let $C$ be a $(-1)$-curve in $D$ meeting $E$. By the definition of $D$ it is contained in some maximal twig of $D$. Consequently, $C\cdot V'=C\cdot W'=0$, hence $C$ is also vertical for $\lambda$. Denote the fiber of $\lambda$ containing it by $F_C$ and let $S_1$ and $S_2$ be the components of $D-E$ meeting $C$. The divisor $C+E+S_1+S_2$ cannot be vertical, because fibers of $\PP^1$-fibrations do not contain branching $(-1)$-curves. Thus, say, $S_1$ is horizontal for $\lambda$, which implies that $C$ has multiplicity $1$ in $F_C$. If so, then it is necessarily a tip of $F_C$, so $S_2$ is horizontal for $\lambda$. But both $S_i$ intersect $V'+W'$, which contradicts the fact that $C$ is a part of some maximal twig of $D$.

(vii) First of all we note that $\psi$ does not contract any component of $D_0$ meeting $E_0$, because the points of intersection with $E_0$ are not snc. It follows that if $A_i$ does not meet $E_i$ then $\psi_{i+1}$ does not change intersections of components of $D_i-E_i$ with $E_i$. Suppose $A_i\cdot E_i=1$ and let $\Delta_{A_i}$ be the connected component of $\Delta_i^-$ intersected by $A_i$. We may assume that $A_i$ meets $\Delta_{A_i}$ in a tip of $D_i$, otherwise the statement is clear. Then $\psi_{i+1}$ contracts exactly $A_i+\Delta_{A_i}$, so it increases (by $1$) the intersection with $E_i$ for only one component of $D_i-E_i$, say $U$, the one which meets $\Delta_{A_i}$. By (vi) $U$ meets at most one $(-2)$-twig of $D_0$, so we are done.

(viii) By \ref{lem:cuspidal_of_gt_and_khalf}(ii) we only need to show that $h^0(m(2K_i+D_i))$ is not affected $\psi_{i+1}$. If $Y$ is smooth and $\sigma:(Y',B')\to (Y,B)$, where $B'=\sigma_*^{-1}B+\Exc \sigma$, is an inner blowup for $B$ then $$2K_{Y'}+B'= \sigma^*(2K_Y+B)+\Exc \sigma,$$ so $h^0(m(2K_Y+B))= h^0(m(2K_{Y'}+B'))$ for every $m\geq 0$. By the definition of $\psi_{i+1}$ it is therefore enough to show that $$h^0(m(2K_i+D_i+A_i))= h^0(m(2K_i+D_i)).$$ But $A_i\cdot(2K_i+D_i)=0$, so $A_i$ is in the fixed part of any linear system $|m(2K_i+D_i)+aA_i|$ with $a>0$.
\end{proof}

\bnot\label{not:tau_eta_etc} Let $q_1,\ldots, q_c$ be the cusps of $\bar E$ and let $j\in \{1,\ldots,c\}$. We define the following numbers:\benum[(i)]

\item $\tau_j\geq 2$ is the number of times $\psi_0$ touches $E$ (equivalently, the number of curves over the cusp $q_j$ contracted by $\psi_0$),

\item $s_j$ is equal to $1$ if $\psi_0$ contains (in a decomposition into blowdowns) a contraction over $q_j$ which is outer for $D-E$ and $0$ otherwise,

\item $\tau_j^*=\tau_j-s_j-1\geq 0$,

\item $\tau^*=\sum_{j=1}^c \tau_j^*$, $s=\sum_{j=1}^c s_j$,  $\tau=\sum_{j=1}^c \tau_j$,

\item $n_k$ for $k=0,1$ is the number of contracted $A_i$'s, i.e.\ the $(-1)$-curves defined in \ref{cor:A_i_exists}, for which $A_i\cdot E_i=k$.

\item $c_0$ and $c_1$ are the numbers of semi-ordinary and non-semi ordinary cusps of $\bar E$ respectively.

\item $\eta=\#\Upsilon_n-\#\Upsilon_0$.\eenum

Clearly, $n=n_0+n_1=n_T+n_N$ and $\#\Upsilon_0=c_0$. By \ref{prop:(Xi,Di)_properties} $n_N\leq \eta\leq n$. \enot

\begin{lemma}\label{lem:properties_of_(Xi,Di)} With the above notation we have: \benum[(i)]

\item $K\cdot (K+D)=p_2(\PP^2,\bar E)$,

\item $\#D_i=\rho(X_i)+i$ for $i\geq 0$,

\item $n+p_2(\PP^2,\bar E)\leq 5$, and if the equality holds then $n\neq 0$, $s=0$ and $D_n$ has no tips.

\item $K_n\cdot (K_n+D_n)=p_2(\PP^2,\bar E)-c-\tau^*-n$,

\item $E_n\cdot (K_n+D_n)=2c-2+\tau^*+n_1$,

\item $(D_n-E_n)\cdot E_n=2c+n_1+\tau^*$,

\item $p_a(D_n)=n+\tau^*+c$.

\eenum\end{lemma}

\begin{proof}  (i) The Riemann-Roch theorem gives $$p_2(\PP^2,\bar E)-h^1(2K+D)+h^2(2K+D)=K\cdot (K+D).$$ By \ref{lem:cuspidal_of_gt_and_khalf}(iv) $(K+D)^-=\Bk D$, so by \ref{lem:Bk}(i) $(K+D)^-$ is an effective $\Q$-divisor with simple normal crossing support and proper fractional coefficients. Since $(K+D)^+$ is nef and big, the Kawamata-Viehweg vanishing theorem (see for example \cite[9.1.18]{Lazarsfeld_II}) says that $h^i(2K+D)=0$ for $i>0$.

(ii) Since $(X,D)$ is a $\Q$-homology plane, we have $\#D=\rho(X)$. By the definition of $X_i$ we have $\#D_i-\rho(X_i)=\#D_{i-1}-\rho(X_{i-1})+1$ for every $i=1,\ldots,n$.

(iii) Let $A_i'$, $i=1,\ldots,n$ denote the proper transform of $A_i$ on $X$ and let $(X,D+\sum_{i=1}^nA_i')\to (\wt X,\wt D)$ be the snc-minimalization of $D+\sum_{i=1}^nA_i'$. Note that if $\kappa_{1/2}(\PP^2,\bar E)\geq 0$ then by \ref{lem:amm_is_snc_min} the boundary $D_n'$ of an almost minimal model is already snc-minimal, so $(\wt X,\wt D)=(X_n',D_n')$. Both pairs are smooth completions of $X_n\setminus D_n$. Note also that it is not possible that $c=1$ and $E$ is a $(-1)$-tip of $D$, because otherwise $|E+C|$, where $C$ is the unique $(-1)$-curve in $D-E$, induces a $\C^*$-fibration of $\PP^2\setminus \bar \E$, an hence $\kappa(\PP^2\setminus\bar E)<2$. Therefore, by the properties of $A_i$'s we see that $(A_i')^2=-1$ and that the minimalization morphism is inner for the boundary, hence it does not change the self-intersection of the log canonical divisor.  The latter is $$(K_X+D+\sum_{i=1}^nA_i')^2=(K_X+D)^2+n=p_2(\PP^2,\bar E)-2+n.$$ The surface $X\setminus D$ is a $\Q$-homology plane of log general type, so it contains no topologically contractible curves. It follows that $\wt X-\wt D$ is of log general type and contains no topologically contractible curves, hence is almost minimal in the sense of \cite{Miyan-OpenSurf}. We have also $\chi(\wt X\setminus\wt  D)=\chi(X\setminus D)$, because $A_i\cap (X_i\setminus D_i)$ is a $\C^*$. By the logarithmic Bogomolov-Miyaoka-Yau inequality, proved originally by Kobayashi-Nakamura-Sakai, (for a formulation convenient for our purposes, which follows from \cite{Langer}, see \cite[2.5]{Palka-exceptional}) we get $$(K_{\wt X}+\wt D)^2+\ind (\wt D)\leq ((K_{\wt X}+\wt D)^+)^2\leq 3\chi(\wt X\setminus \wt D)=3,$$ so \beq p_2(\PP^2,\bar E)+n+\ind(\tilde D)\leq 5.\label{eq:BMY_for_Dtilda}\eeq If $\ind(\tilde D)=0$ then necessarily $(K_{\wt X}+\wt D)^-=0$, so $\wt D$ (and hence $D_n'$) has no maximal twigs, hence no tips. Equivalently, $s_j=0$ for $j=1,\ldots,c$ (the twigs of $D$ contracted by $\psi_0$ are disjoint from all $A_i'$) and $D_n$ has no tips. Finally, $D_0'=D$ has at least three maximal twigs, so if $n=0$ then $p_2(\PP^2,\bar E)\neq 5$.

(iv) Since $\psi_i$ is a composition of blowdowns which are inner with respect to $D_{i-1}+A_{i-1}$, we have $$\psi_i^*(K_i+D_i)=K_{i-1}+D_{i-1}+A_{i-1},$$ so $K_i\cdot (K_i+D_i)=K_{i-1}\cdot (K_{i-1}+D_{i-1})-1$ and we get $K_n\cdot(K_n+D_n)=K_0\cdot(K_0+D_0)-n$. From the definition of $\psi_0$ we compute $K_0\cdot (K_0+D_0-E_0)=K\cdot (K+D-E)+s$ and $K_0\cdot E_0=K\cdot E-\tau$, so $$K_0\cdot (K_0+D_0)=K\cdot(K+D)-c-\tau^*=p_2(\PP^2,\bar E)-c-\tau^*.$$

(v) Using the remark on $\psi_i$ from (iv) we have $$E_i\cdot(K_i+D_i)=E_{i-1}\cdot (K_{i-1}+D_{i-1})+E_{i-1}\cdot A_{i-1},$$ so $E_n\cdot (K_n+D_n)=E_0\cdot (K_0+D_0)+n_1$. Now $E_0\cdot (K_0+D_0-E_0)=E\cdot (K+D-E)-s$ and $E_0^2=E^2+\tau$, so $$E_0\cdot(K_0+D_0)=E\cdot(K+D)+\tau^*+c.$$ Finally, $E\cdot (K+D)=\beta_D(E)-2=c-2$.

(vi) $(D_n-E_n)\cdot E_n=(D_0-E_0)\cdot E_0+n_1=\ds \sum_{j=1}^t(\tau_j+1-s_j)+n_1=\tau^*+2c+n_1$.

(vii) Elementary properties of the arithmetic genus give $p_a(D_n)=p_a(D_0)+n$ and $p_a(D_0)=p_a(D_0-E_0)+(D_0-E_0)\cdot E_0-1$, so $$p_a(D_0)=(1-c)+\sum_{j=1}^c(\tau_j+1-s_j)-1=\tau^*+c.$$
\end{proof}

We now prove one of the bounds in Theorem \ref{thm:MAIN4}.

\begin{proof}[Proof of Theorem \ref{thm:MAIN4}, part 1] By \eqref{eq:BMY_for_Dtilda} $\ind(\tilde D)\leq 5-p_2(\PP^2,\bar E)-n$. Note that every $\pi_0(A_i)\subseteq \PP^2$, $i>1$ meets $\pi_0(A_1)$, and by the definition of $A_i$'s the intersection point is necessarily one of the cusps of $\bar E$. Therefore, in the process of minimalization at least $c-n-1$ cusps are untouched. It is a straightforward computation (see \cite[4.2]{OrZa_cusp}) that the contribution to the inductance of the boundary from the twigs over any cusp is strictly bigger than $\frac{1}{2}$, hence $\frac{1}{2}(c-n-1)<\ind(\tilde D)$, which gives $c+n+2p_2(\PP^2,\bar E)\leq 10$. On the other hand, applying the BMY inequality to $(X,D)$ we have $\frac{1}{2}c<\ind(D)\leq 5-p_2(\PP^2,\bar E)$, so $c+2p_2(\PP^2,\bar E)\leq 9$. These inequalities may be written together as
\beq c+2p_2(\PP^2,\bar E)+\max(0,n-1)\leq 9\label{eq:bound_on_c+n}.\eeq \end{proof}

We now want to compute the square of $2K_n+D_n^\flat$. We define $\Upsilon_i^0$ to be the sum of these components of $\Upsilon_i$ which do not meet $\Delta_i^+$. By the definition of $\Upsilon_i$, every component $U$ of $\Upsilon_i^0$ if a $(-1)$-curve meeting exactly one component of $D_i$ and has $\beta_{D_i}(U)=2$ (so it is tangent to this component or meets it twice transversally). Clearly, $(\psi_{i+1})_*\Upsilon_i^0\leq \Upsilon_{i+1}^0$ and $\Upsilon_0^0$ consists of the $C_i$'s for which $q_i\in\bar E$ is an ordinary cusp.

\bsrem Note that $\Upsilon_{i+1}^0$ may contain more components than just $(\psi_{i+1})_*\Upsilon_i^0$. To see this let $U_1+\ldots+U_k$ be the $(-2)$-twig of $D_i$ met by $A_i$ and let $U_{k+1}$ be the component of $D_i-U_1-\ldots-U_k$ met by $A_i$. If, for instance, $A_i\cdot U_1=U_k\cdot U_{k+1}=1$ (so $A_i+U_1+\ldots+U_k+U_{k+1}$ is an snc cycle) then $\psi_{i+1}$ contracts exactly $A_i+U_1+\ldots+U_{k-1}$ and $(\psi_{i+1})_*U_k$ is a $(-1)$-curve meeting $U_{k+1}$ in two different points, hence it is a component of $\Upsilon_{i+1}^0$. Similarly, if $U_{k+1}$ and $U_1+\ldots+U_k$ are maximal twigs of $D_i$ meeting a common component $U_{k+2}$ then $(\psi_{i+1})_*U_{k+1}$ is a component of $\Upsilon_{i+1}^0-(\psi_{i+1})_*\Upsilon_i^0$. In both examples $\psi_{i+1}$ is of type $I$, but this is not true in general. \esrem

\blem\label{lem:square_of_2K+Dflat} Let $\bar E\subseteq \PP^2$, $(X_n,D_n)$ and $\Upsilon_n^0$ be as above. Let $(X_n',\frac{1}{2}D_n')$ be an almost minimal model as in \ref{def:psi_and_models} and $\Delta_n'$ the sum of maximal $(-2)$-twigs of $D_n'$. Then \beq\label{eq:square_of_2K+Dflat} (2K_n+D_n^\flat)^2+\delta(\Delta^-_n)=3p_2(\PP^2,\bar E)+8+b_0(\Delta_n')+\#\Upsilon_n^0-\rho(X_n')-n.\eeq
\elem

\begin{proof} We have $(\Upsilon_n+\Delta_n)\cdot (2K_n+D^\flat)=0$ and $([1,2,\ldots,2])^2=-1$, so by \ref{lem:Bk} and \ref{lem:properties_of_(Xi,Di)}(vi) $$(2K_n+D_n^\flat)^2=(2K_n+D_n)^2-(\Upsilon_n+\Delta_n^+)^2-(\Bk'\Delta_n^-)^2=(2K_n+D_n)^2+\#\Upsilon_n+\ind (\Delta_n^-).$$ We have $\#\Upsilon_n=c_0+\eta$ and we compute easily $\ind (\Delta_n^-)=b_0(\Delta^-_n)-\delta(\Delta^-_n)$. Now
\begin{align*}
(2K_n+D_n)^2 &=K_n^2+3K_n\cdot(K_n+D_n)+D_n\cdot(K_n+D_n)\\
&=K_n^2+3(p_2(\PP^2,\bar E)-c-\tau^*-n)+2(n+\tau^*+c-1) \\
&= K_n^2+3p_2(\PP^2,\bar E)-2-c-\tau^*-n .
\end{align*} The Noether formula gives $$(2K_n+D_n^\flat)^2=3p_2(\PP^2,\bar E)+8+b_0(\Delta^-_n)-\delta(\Delta^-_n)-\rho(X_n)-\tau^*-(n-\eta)-c_1.$$

Fix $j\in \{1,\ldots, c\}$. Let $C_j$ be the $(-1)$-curve in $\wt Q_j$ meeting $E_0$. If $s_j\neq 1$ then there is exactly one more component of $\wt Q$ meeting $E_0$, call it $B_0$. It meets $E_0$ transversally. Put $B_0=0$ if $s_1=1$. Let $A,B$ and $A', B'$ be the proper transforms of $A_0, B_0$ on $X_n$ and on $X_n'$ respectively. Clearly, $A$ and $B$ are not contained in $(-2)$-twigs of $D_0$. We have $(A')^2=A^2-\tau$ and $(B')^2=B^2-1$. If $A'$ is a component of $\Delta_n'$ then $A^2=\tau-2\geq 0$, so $\psi$ touches $A_0$. But then $A$ is not contained in a twig of $D_n$, which is impossible. If $B'$ is a component of $\Delta_n'$ then $B$ is a $(-1)$-curve non-branching in $D_0-E_0$, which is impossible by the definition of $\psi_0$. Thus $A', B'$ are not components of $\Delta_n'$. We get $b_0(\Delta_n')=b_0(\Delta_n)+s$. By the definition of $\Delta_i^+$ we have $\#\Upsilon_i-b_0(\Delta_i^+)=\#\Upsilon_i^0$. Then $$b_0(\Delta_n')-s=b_0(\Delta_n^-)+b_0(\Delta_n^+)=b_0(\Delta_n^-)+\#\Upsilon_n-\#\Upsilon_n^0.$$ We have $\#\Upsilon_n=c_0+\eta$ and $\rho(X_n')=\rho(X_n)+\tau$, hence $$(2K_n+D_n^\flat)^2=3p_2(\PP^2,\bar E)+8+b_0(\Delta_n')+\#\Upsilon_n^0-\delta(\Delta^-_n)-\rho(X_n')-n.$$
\end{proof}

By \ref{lem:properties_of_(Xi,Di)}(iii) $p_2(\PP^2,\bar E)\leq 4$. Recall that $\alpha_n$ is the morphism defined on $X_n$ which contracts $\Upsilon_n+\Delta_n$. An \emph{open log del Pezzo surface of rank $1$} is a log surface $(X,B)$, such that $-(K_X+B)$ is ample, $B\neq 0$ and $\rho(X)=1$. We summarize our analysis in the following theorem, which implies \ref{thm:MAIN1}. (To recover \ref{thm:MAIN1}(4) use \eqref{eq:b2_bound} and \ref{thm:MAIN4}).

\bthm\label{thm:MAIN1detailed} Let $\bar E\subseteq \PP^2$ be a rational cuspidal curve for which $\PP^2\setminus \bar E$ is of log general type and let $(X_0,D_0)\to (\PP^2,\bar E)$ be the minimal weak resolution of singularities. Let $(X',\frac{1}{2}D')$ and $(Y,\frac{1}{2}D_Y)$ denote an almost minimal and a minimal model of $(X_0,\frac{1}{2}D_0)$. Put $n=p_a(D')$ Then: \benum[(1)]

\item $X'\setminus D'\cong Y\setminus D_Y$ is isomorphic to an open subset of $\PP^2\setminus \bar E\cong X_0\setminus D_0$ with the complement being a disjoint sum of $n$ curves isomorphic to $\C^*$.

\item $\#D'=\rho(X')+n$ and $K'\cdot D'=p_2(\PP,\bar E)+\rho(X')-n-10$.

\item $n+p_2(\PP^2,\bar E)\leq 5$, and if the equality holds then $n\neq 0$ and $D'$ has no twigs.

\item If $\kappa_{1/2}(\PP^2,\bar E)=-\8$ then one of the following holds: \benum[(a)]
    \item $Y$ admits a $\PP^1$-fibration with irreducible fibers inducing a $\C^{**}$-fibration on $\PP^2\setminus\bar E$.
    \item $(Y,\frac{1}{2}D_Y)$ is an open log terminal log del Pezzo surface of rank $1$, whose boundary has $n+1\leq 6$ components.
     \eenum

\item If $\kappa_{1/2}(\PP^2,\bar E)\geq 0$ then $(X',D')$ is snc-minimal and $n\leq 4$.

\item If $\PP^2\setminus \bar E$ does not admit a structural $\C^{**}$-fibration then $$\rho(X')\leq 3p_2(\PP^2,\bar E)+8-n+b_0(\Delta_0')+c_0',$$ where $\Delta_0'$ is the sum of maximal $(-2)$-twigs of $D$ and $c_0'$ is the number of ordinary cusps of $\bar E$.
\eenum\ethm

\begin{proof} We have $(Y,D_Y)=(Y_n,D_{Y_n})$ and $(X',D')=(X_n',D_n')$. (1) Recall (\ref{rem:diagram}) that we have a commutative diagram $$ \xymatrix{
 {} &{(X,D)}\ar@{>}[r]^-{\psi'}\ar@{>}[d]^{\psi_0}\ar@{>}[ld]_-{\pi} &{(X_n',D_n')}\ar@{>}[d]^{\varphi_n}& {}\\
{(\PP^2,\bar E)} &{(X_0,D_0)}\ar@{>}[r]^-{\psi}\ar@{>}[l]^-{\pi_0} &{(X_n,D_n)}\ar@{>}[r]^{\alpha_n}& {(Y_n,D_{Y_n})}\ .}$$
The morphism $\alpha_n\circ\varphi_n\:(X_n',D_n')\to (X_n,D_n)\to (Y_n,D_{Y_n})$ is the minimal log resolution of singularities, so $X_n'\setminus D_n'\cong X_n\setminus D_n\cong Y_n\setminus D_{Y_n}$. By \ref{cor:A_i_exists} $A_i\cap (X_i\setminus D_i)\cong \C^*$. Also, $p_a(D')=p_a(D)+n=n$.

(2) As in the proof of \ref{lem:properties_of_(Xi,Di)}(iv) we have $$K_n'\cdot(K_n'+D_n')=K\cdot (K+D)=p_2(\PP^2,\bar E),$$ so by the Noether formula $$K_n'\cdot D_n'=p_2(\PP^2,\bar E)+\rho(X_n')-n-10.$$ By \ref{lem:properties_of_(Xi,Di)}(ii) $\#D_n'-\rho(X_n')=\#D_n-\rho(X_n)=n$.

(3) follows from  \ref{lem:properties_of_(Xi,Di)}(iii) and from the fact that $D_n'$ has no twigs if and only if $s=0$ and $D_n$ has no twigs.

(4) The pair $(Y_n,\frac{1}{2}D_{Y_n})=(\alpha_n(X_n),\frac{1}{2}\alpha_{n*}D_n)$ is a log Mori fiber space by \ref{cor:A_i_exists}. Because of the affiness of $\PP^2\setminus \bar E$ the divisor $D_{Y_n}$ is nonzero. If the base of the fiber space is a point then $Y_n$ is a log del Pezzo surface of rank $1$. Its singularities are log terminal by \ref{lem:producing_A} and $\#D_{Y_n}-\rho(Y_n)=\#D_n-\rho(X_n)=n$ by \ref{lem:properties_of_(Xi,Di)}(ii). Assume the base is a curve and let $f$ be a general fiber. Then we have $f\cdot (K_{Y_n}+\frac{1}{2}D_{Y_n})<0$, so $f\cdot D_{Y_n}\leq 3$. If $f\cdot D_{Y_n}\leq 2$ then $$\kappa(\PP^2\setminus\bar E)=\kappa(Y_n\setminus D_{Y_n})\leq 1+\kappa(\C^*)=1$$ by the 'easy addition theorem'.

(6) If $\PP^2\setminus \bar E$ is not $\C^{**}$-fibered then either $K_{Y_n}+\frac{1}{2}D_{Y_n}$ is nef or the minimal model $(Y_n,\frac{1}{2}D_{Y_n})$ is a log del Pezzo surface, and then $-(K_{Y_n}+\frac{1}{2}D_{Y_n})$ is ample. In both cases $(K_{Y_n}+\frac{1}{2}D_{Y_n})^2\geq 0$, which by \ref{lem:producing_A}(i) and \ref{lem:square_of_2K+Dflat} gives $$\rho(X_n')\leq 3p_2(\PP^2,\bar E)+8-n+b_0(\Delta'_n)+\#\Upsilon_n^0-\delta(\Delta^-_n),$$ and the inequality is strict in the del Pezzo case. Put $\upsilon_n=\#\Upsilon_n^0$.  Note that $\psi_*(\Upsilon_0^0)\leq \Upsilon_n^0$, so $c_0'=\upsilon_0\leq \upsilon_n$. It is easy to see that if a component of $\Upsilon_{i+1}^0-\psi_*(\Upsilon_i^0)$ is created then the tip of $\Delta_i^-$ met by $A_i$ is a tip of $D_i$. \beq b_0(\Delta_n')+(\upsilon_n-\upsilon_0)\leq b_0(\Delta_0').\label{eq:b0(Delta_n)_bound}\eeq Then $b_0(\Delta_n')+\upsilon_n\leq b_0(\Delta_0')+c_0'$, so $$\rho(X_n')\leq 3p_2(\PP^2,\bar E)+8-n+b_0(\Delta'_0)+c_0'-t_\Delta(\psi').$$

(5) Assume $\kappa_{1/2}(\PP^2,\bar E)\geq 0$. Then $(X_n',D_n')$ is snc-minimal by \ref{lem:amm_is_snc_min}. Suppose $n=5$. Then $p_2(\PP^2,\bar E)=b_0(\Delta_n')=c_0'=0$ by \eqref{eq:BMY_for_Dtilda}, so (6) gives $\rho(X_5)+\tau=\rho(X_5')\leq 3$. But $\tau\geq 2c\geq 2$, so $X_5\cong \PP^2$. By \ref{lem:properties_of_(Xi,Di)}(ii) $\#D_5=6$, so $\deg (2K_{\PP^2}+D_5)\geq 0$, hence $2K_5'+D_5'\geq 0$; a contradiction.
\end{proof}

\vskip 0.3cm
We now prove Theorems \ref{thm:MAIN2} and \ref{thm:MAIN4}

\begin{proof}[Proof of Theorem \ref{thm:MAIN2}(1)]  By \ref{thm:MAIN1detailed}(6) $$\rho(X')\leq 3p_2(\PP^2,\bar E)+8-n+b_0(\Delta_0')+c_0'$$ and the inequality is strict in the del Pezzo case.  Now $b_0(\Delta_0')+c_0'$ can be bounded from above by the numer of maximal twigs of $D$, which in not bigger than $17-p_2(\PP^2,\bar E)$ by a result of Tono \cite{Tono-on_the_number_of_cusps}. But it is better, and easier, to reuse the BMY inequality as follows (we owe this observation to Mariusz Koras). First, observe that if we have a twig $R=[a_1,\ldots,a_n]$ with $a_i\geq 2$ and its subtwig $R'=[a_1,\ldots,a_{n-1}]$ then $\ind(R)\geq \ind(R')$. Indeed, this follows by induction from the formula $\ind([a_1,\ldots,a_n])=1/(a_1-\ind([a_2,\ldots,a_n]))$. We get $$\frac{1}{2}b_0(\Delta_0')+\frac{1}{3}c_0'=b_0(\Delta_0')\ind([2])+c_0'\ind([3])\leq \ind(D)\leq 5-p_2(\PP^2,\bar E),$$ where the last inequality is just the BMY inequality for $(X,D)$ (cf. \eqref{eq:BMY_for_Dtilda}). Thus \beq b_0(\Delta_0')\leq 10-2p_2(\PP^2,\bar E)-\frac{2}{3}c_0'\label{eq:b0(Delta)_bound}.\eeq We obtain:

\beq\label{eq:b2_bound} \rho(X_n')\leq p_2(\PP^2,\bar E)+18-n+\frac{1}{3}c_0'. \eeq  Then \ref{thm:MAIN1detailed}(2) gives
\begin{eqnarray}
    \label{eq:D_n'_bound}\#D_n'&\leq& p_2(\PP^2,\bar E)+18+\frac{1}{3}c_0',\\
     K_n'\cdot D_n'&\leq& 2p_2(\PP^2,\bar E)+8-2n+\frac{1}{3}c_0',
\end{eqnarray}
and the inequalities are strict in the del Pezzo case. It remains to prove that \beq p_2(\PP^2,\bar E)+\frac{1}{3}c_0'\leq 4,\label{eq:p2+c0'}\eeq because then $\rho(X_n')\leq 22$, $\#D_n'\leq 22$ and $K_n'\cdot D_n'\leq 16$.

Suppose $p_2(\PP^2,\bar E)>4-\frac{1}{3}c_0'$. By \eqref{eq:bound_on_c+n} $p_2(\PP^2,\bar E)\leq \frac{9}{2}-\frac{1}{2}c<5$, so $c_0'>0$ and $(c-c_0')+\frac{1}{3}c_0'\leq 1$. Then $c=c_0'\leq 3$ and $p_2(\PP^2,\bar E)=4$. We get $\frac{5}{6}c_0'\leq \ind(D)\leq 5-p_2(\PP^2,\bar E)=1$, so $c_0'=1$. The genus formula gives $\binom{\deg \bar E-1}{2}=\frac{1}{2}\mu(\mu-1)$, where $\mu=2$, so $\deg \bar E=3$. However, $2K_{\PP^2}+\bar E=\pi_*(2K_X+D)\geq 0$; a contradiction.
\end{proof}

\begin{proof}[Proof of Theorem \ref{thm:MAIN4} (part 2)] By \eqref{eq:bound_on_c+n} we only need to show that $c\leq 6$. The case when $\PP^2\setminus \bar E$ has a structural $\C^{**}$-fibration will be treated in \ref{lem:Cstst_fibrations}(iv), so we assume the opposite. Over each cusp $\psi_0$ contracts at least two curves and makes one component tangent to $E_0$, hence this component is not contracted by $\psi\: X_0\to X_n$. It follows that $D_n'-E_n'$ contains at least $3c$ components. Therefore, by \eqref{eq:D_n'_bound} $1+3c\leq \#D_n'\leq p_2(\PP^2,\bar E)+18+\frac{c_0'}{3}$, so \beq 3c\leq p_2(\PP^2,\bar E)+17+\frac{c_0'}{3}.\label{eq:bound_on_c}\eeq Suppose $c\geq 7$. By \eqref{eq:p2+c0'} the above inequality gives $c=7$ and $p_2(\PP^2,\bar E)=4-\frac{1}{3}c_0'$. The BMY inequality for $(X,D)$ gives $\frac{5}{6}c_0'\leq \ind(D)\leq 5-p_2(\PP^2,\bar E)=1+\frac{1}{3}c_0'$, so $c_0'\leq 2$. Then in fact $c_0'=0$, $p_2(\PP^2,\bar E)=4$ and $\ind(D)\leq 1$. We arrive at a contradiction, because the contribution to $\ind(D)$ of each cusp is strictly bigger than $\frac{1}{2}$. This finishes the proof of Theorem \ref{thm:MAIN4}.
\end{proof}

\begin{proof}[Proof of Theorem \ref{thm:MAIN2}(2)] Assume $\kappa_{1/2}(\PP^2,\bar E)=-\8$. By part (1) of the Theorem $\rho(X_n')\leq 21$, $\#D_n'\leq 21$ and $K_n'\cdot D_n'\leq 15$. Let $V'$ be a component of $D_n'$ and $V$ its image in $D_n$. We show that $(V')^2\geq -3$. If $V=0$ or if $V\leq \Delta_n$ then $(V')^2=-1$ or $-2$. If $V\leq \Upsilon_n$ and $\varphi_n\:X_n'\to X_n$ touches $V'$ then $V'$ is a component of a resolution of a semi-ordinary cusp, so $(V')^2\geq -3$. We may therefore assume $V\leq R:=D_n-\Upsilon_n-\Delta_n$. Since $2K_n+D_n^\flat$ is anti-ample off $\Upsilon_n+\Delta_n$, we get $0>-V^2+V\cdot(D_n^\flat-V)+2V\cdot (K_n+V)$, so since $V\cdot (D_n^\flat-V)\geq \beta_R(V)$, $$V^2-\beta_{R}(V)\geq -3.$$ We may therefore assume that $\varphi_n$ touches $V'$, otherwise $(V')^2=V^2\geq -2$ and we are done. Then $V$ is tangent to $E_n$ and all (three) components of $D_n$ passing through the point of tangency belong to $R$. Thus each blowup constituting $\varphi_n$ whose center is on the proper transform of $V$ decreases the self-intersection and the branching number (computed with respect to the reduced total transform of $R$) by $1$, hence $(V')^2\geq -2$.

Assume $\kappa_{1/2}(\PP^2,\bar E)\geq 0$. Let $V'$ be a component of $D_n'$. Put $\alpha=(V')^2$. If $\alpha\geq 0$ then $V'$ intersects every effective divisor non-negatively, so $0\leq V'\cdot (2K_n'+D_n')=-\alpha-4+\beta_{D_n'}(V').$ This gives $$\alpha\leq \max(0,\beta_{D_n'}(V')-4).$$ If $V'=E_n'$ then write $D_n'-V'=\sum V_j$, otherwise write $D_n-V'-E_n'=\sum_jV_j'$, where $V_j'$ are irreducible. Let $V$ and $V_j$ be respectively the proper transforms of $V'$ and $V_j'$ on $X$.  It follows from the definition of $\psi_i$ that the branching number of the image of $V'$ increases at most by one under $\psi_i$ and that it increases exactly by one if and only if $A_{i-1}$ meets this image. Therefore, $\beta_{D'}(V')\leq \beta_D(V)+n$ and $\sum (\beta_{D_n'}(V_j')-\beta_{D}(V_j))\leq 2n$, where the summation can be taken over any subset of irreducible components of $\sum_jV_j'$. If $V'=E_n'$ then by \eqref{eq:bound_on_c+n} we get $$(E_n')^2\leq \max(0,c+n-4)\leq \max(0,6-2p_2(\PP^2,\bar E)).$$
If $V'\neq E_n'$ then, as a consequence of the fact that every component of $D-E$ contracts to a smooth point, $\beta_D(V)\leq 3$, so $\alpha\leq \max(0,n-1)\leq 4$.

Put $\alpha_j=(V_j')^2$. Let $I$ and $J$ be the sets of these $V_j'$ for which $\alpha_j=-1$ and $\alpha_j\geq 0$ respectively. Then $$-\sum K_n'\cdot V_j'=\sum_j(\alpha_j+2)\leq \#I+\sum_{J}(\beta_{D_n'}(V')-2)\leq \#I+\sum_{J}(\beta_{D}(V_j)-2)+2n.$$ Since $\beta_D(V_j)\leq 3$ for every $V_j$, the right hand side of the inequality is bounded by $\#I+\#J+2n$.

Assume $V\neq E_n'$. By the genus formula $-\alpha-2-(E_n')^2-2=K_n'\cdot (V'+E_n')=K_n'\cdot D_n'-\sum_j K_n'\cdot V_j'$, so the above inequality gives $-\alpha\leq (E_n')^2+4+K_n'\cdot D_n'+(\#D_n'-2)+2n$. By \eqref{eq:b2_bound} and \eqref{eq:D_n'_bound} this results with the inequality $$-\alpha\leq (E_n')^2+3p_2(\PP^2,\bar E)+28+\frac{2}{3}c_0'.$$ Now for $p_2(\PP^2,\bar E)\neq 4$ we obtain $-\alpha\leq 6+p_2(\PP^2,\bar E)+28+\frac{2}{3}c_0'\leq 38+\frac{1}{3}c_0'\leq 40$, because $c_0'\leq c\leq 6$. If $p_2(\PP^2,\bar E)=4$ then $c_0'=0$ by \eqref{eq:p2+c0'}, so $-\alpha\leq 3p_2(\PP^2,\bar E)+28=40$.

Assume $V=E_n'$. As above we obtain $-\alpha\leq 2+K_n'\cdot D_n'+(\#D_n'-1)+2n$, hence $-\alpha\leq 3p_2(\PP^2,\bar E)+27\leq 39$.

\end{proof}

\brem Note that if $V,W$ are distinct components of $D_n'$ then, because $D_n'$ is an snc-divisor of arithmetic genus $n$, we have $0\leq V\cdot W\leq n+1\leq 6$. In particular, \ref{thm:MAIN2}(2) shows that if $\kappa_{1/2}(\PP^2,\bar E)\geq 0$ then the size and all entries of the intersection matrix of $D_n'$ are bounded. In case $\kappa_{1/2}(\PP^2,\bar E)=-\8$ we have no upper bound for diagonal entries. \erem

We make the following negativity conjecture for cuspidal curves.

\begin{conjecture}\label{con:k_half_vanishes} If $\bar E\subseteq \PP^2$ is a rational cuspidal curve then $\kappa_{1/2}(\PP^2,\bar E)=-\8$. \end{conjecture}

\brem\label{rem:rigidity} By the definition of $\kappa_{1/2}$ the conjecture simply says that for the minimal log resolution $(X,D)\to (\PP^2,\bar E)$ one has $h^0(m(2K_X+D))=0$ for every $m>0$. It implies both the weak rigidity conjecture and the Coolidge-Nagata conjecture (see \ref{con:conjectures}). Theorem $\ref{thm:MAIN2}$ shows that almost minimal models of potential counterexamples to \ref{con:k_half_vanishes} satisfy strong combinatorial restrictions. Also, the minimalization process is well described and has at most four steps, which gives a reasonable control over $D$.
\erem

We finish this section with the proof of Corollary \ref{thm:MAIN3} in case $\PP^2\setminus\bar E$ does not admit a structural $\C^{**}$-fibration.

\begin{proof}[Proof of Corollary \ref{thm:MAIN3} (part 1)] The divisor $D$ is a rational tree, so the core graph and the Eisenbud-Neumann diagram are trees, which implies that a universal bound on the number of vertices gives a bound on the number of possible graphs. For a reduced divisor $B$ let $\core(B)$ and $\vc(B)$ denote the core and respectively the number of vertices in the core graph of $B$. Write $t(B)$ for the number of maximal twigs of $B$. Let $t_\Delta(\psi')$ and $t(\psi')$ be respectively the number of maximal $(-2)$-twigs of $D$ contracted completely by $\psi'$ and the number of all maximal twigs of $D$ contracted completely by $\psi'$. Recall that $\upsilon_i=\#\Upsilon_i^0$. By the proof of \ref{thm:MAIN1detailed}(6) $\vc(D_n')\leq \#D_n'=\rho(X_n')+n\leq 3p_2(\PP^2,\bar E)+8+b_0(\Delta'_n)+\upsilon_n$. We claim that $$(\vc(D_i')-\vc(D_{i+1}'))+(\upsilon_{i+1}-\upsilon_i)\leq (b_0(\Delta_i')-b_0(\Delta_{i+1}'))+1.$$ By \ref{prop:(Xi,Di)_properties}(iv)-(v) $\upsilon_{i+1}-\upsilon_i\leq 1$. We may assume that $A_i$ meets $\Delta_i'$ in a tip of $D_i'$, otherwise $\psi_{i+1}'$ does not contract completely any maximal twig of $D_i'$, so $\vc(D_i')-\vc(D_{i+1}')\leq 0$ and the claim holds. Then the above inequality is equivalent to $(\vc(D_i')-\vc(D_{i+1}'))+(\upsilon_{i+1}-\upsilon_i)\leq 2.$ Since $\psi_{i+1}'$ contracts completely at most two maximal twigs of $D_i$, we have $\vc(D_i')-\vc(D_{i+1}')\leq 2$. But if equality holds then the maximal twigs of $D_i'$ contracted by $\psi_{i+1}'$ meet different components of the remaining part of $D_i'$, so $(\psi_{i+1})_*\Upsilon_i^0=\Upsilon_{i+1}^0$ and hence $\upsilon_{i+1}-\upsilon_i=0$. Thus the claim holds. It follows that $$\vc(D)-\vc(D_n')+\upsilon_n-c_0'\leq b_0(\Delta_0')-b_0(\Delta_n')+n,$$ and hence that $$\vc(D)\leq 3p_2(\PP^2,\bar E)+8+b_0(\Delta_0')+n+c_0'.$$ By \eqref{eq:b0(Delta)_bound} $$\vc(D)\leq p_2(\PP^2,\bar E)+18+n+\frac{1}{3}c_0'$$ and by \eqref{eq:BMY_for_Dtilda} $p_2(\PP^2,\bar E)+n\leq 5-\ind(\tilde D)\leq 5-\frac{5}{6}c_0'$. Then $\vc(D)\leq 23-\frac{1}{2}c_0'\leq 23$. Since $D$ has at least three maximal twigs, $\#\core(D)\leq 20$.

\end{proof}

\section{The structural $\C^{**}$-fibration}\label{sec:Cstst}

In this section we prove Corollary \ref{thm:MAIN3} and Theorem \ref{thm:MAIN4} in case $\PP^2\setminus\bar E$ admits a structural $\C^{**}$-fibration. By definition, there is a $\PP^1$-fibration $\xi_Y\:Y_n\to \PP^1$ with irreducible fibers, such that $D_{Y_n}$ meets a general fiber exactly three times. The morphism $\alpha_n\:(X_n,D_n)\to (Y_n,D_{Y_n})$ decomposes as $$(X_n,D_n)\xrightarrow{\alpha_n^+} (Z,D_Z)\xrightarrow{\alpha_n^-} (Y_n,D_{Y_n}),$$ where $\alpha_n^+$ contracts $\Upsilon_n+\Delta_n^+$, $\alpha_n^-$ contracts $\Delta_n^-$ and $D_Z=(\alpha_n^+)_*D_n$. It follows that $Z$ is smooth. We denote the $\PP^1$-fibration $\xi_Y\circ\alpha_n^-$ induced on $Z$ by $\xi$. Note that even if $F$ is a smooth fiber of $\xi$, its affine part $F\cap (Z\setminus D_Z)$ might be a singular fiber of $\xi_{|Z\setminus D_Z}$, i.e.\ might be non-isomorphic to $\C^{**}$ as a scheme.

\blem\label{lem:Cstst_fibrations} Assume $\PP^2\setminus\bar E$ admits a structural $\C^{**}$-fibration. Let $(Z,D_Z)$ and $\xi\:Z\to\PP^1$ be as above. Let $h$ be the number of $\xi$-horizontal components of $D_Z$, $\nu$ - the number of fibers of $\xi$ contained in $D_Z$ and $\sigma$ - the number of singular fibers of $\xi_{Z\setminus D_Z}$. Then:\benum[(i)]

\item $\nu=n+2-h$,

\item $\sigma=h+1-n$,

\item every singular fiber of $\xi$ has a dual graph as below, where black dots stand for horizontal components of $D_Z$ and are joined by dotted lines with components meeting them:$$\xymatrix{{\bullet}\ar@{.}[r]&  {-2}\ar@{-}[r] &{-2}\ar@{-}[r]\ar@{-}[d] &{-2}\\
{}&{\bullet}\ar@{.}[r]&{-1}&{}}.$$

\item $\bar E$ has at most three cusps.
\eenum
\elem

\vskip 0.3cm
\begin{proof} Let $H$ be the divisor consisting of components of $D_Z$ which are horizontal for $\xi$.

(iii) Let $f$ be a singular fiber of $\xi$. Put $F=f_{red}$. Denote by $L_f\leq F$ the proper transform of $\alpha_{n*}F$. Since the support of $F-L_f$ is contained in the support of $\Exc \alpha_n^-=\Delta_n^-$, we see that $F-L_f$ consists of $(-2)$-curves and is contained in the sum of maximal twigs of $D_Z$. Now $F$, as a fiber of a $\PP^1$-fibration of a smooth surface, is an (snc-) tree containing a $(-1)$-curve, so $L_f$ is the $(-1)$-curve. It follows that either $F$ is a chain of type $[2,1,2]$ or $F-L_f$ is a chain of type $[2,2,2]$ and $L_f$ meets the middle $(-2)$-curve. Suppose that $F=[2,1,2]$. Denote the $(-2)$-tips by $T_1, T_2$. The multiplicity $\mu(L_f)$ of $L_f$ in $f$, equals $2$. Consider the case when $L_f$ is a component of $D_Z$. If $H\cdot L_f=0$ then, say, $H\cdot B_1\geq 2$, so $B_1$ is not contained in a twig of $D_Z$; a contradiction. Thus $H$ meets $L_f$ and hence $L_f\cdot H=1$. Then $H$ meets exactly one $(-2)$-tip of $F$, say $B_1$, so the proper transform of $B_2$ on $X_n$ is contained in $\Delta_n^+$. This is impossible in view of the definition of $\alpha_n^+$; a contradiction. Consider the case when $L_f$ is not a component of $D_Z$. Because of the connectedness of $D_Z$, $H$ meets both $(-2)$-tips of $F$. The multiplicity of both tips in $f$ is one, so since they are part of $\Delta_n^-$, $H$ meets each of them once and transversally. Thus $H$ meets $L_f$, so $H\cdot f\geq 2+\mu(L_f)=4$; a contradiction.

Therefore $F-L_f$ is of type $[2,2,2]$ and $L_f$ meets the middle $(-2)$-curve. The surface $\PP^2\setminus \bar E$ is assumed to be of log general type, so it contains no $\C^1$ by \ref{lem:cuspidal_of_gt_and_khalf}(iv). This means that $H$ meets $L_f$. Since $L_f$ and the middles $(-2)$-curve in $F-L_f$ have multiplicity $2$ in $f$, the latter does not meet $H$.

(i) Each fiber of $\xi$ contains at most one component not contained in $D_Z$, so $\rho(Z)=2+(\#(D_Z-H)-\nu)$. We have $\rho(Z)=\#D_Z+n$, so we obtain $h+\nu=n+2$.

(ii) Put $\pi=\xi_{|Z\setminus D_Z}$. The Euler characteristic of $Z\setminus D_Z$ is $1$, the Euler characteristic of a base of $\pi$ is $\chi(\PP^1)-\nu=2-\nu$ and the Euler characteristic of a general fiber is $\chi(\C^{**})=-1$. The Suzuki formula (\cite[3.1.8]{Miyan-OpenSurf}) reads as $$1=(2-\nu)(-1)+\sum(\chi(F_s)+1),$$ where the sum is taken over singular fibers of $\pi$. Now by (iv) we see that if $\bar F_s$, the fiber of $\xi$ containing $F_s$, is singular then $F_s\cong \C^*$. But the latter holds also when $\bar F_s$ is smooth, because the only other option is $F_s\cong \C^1$, which is impossible by \ref{lem:cuspidal_of_gt_and_khalf}(iii). Thus $\chi(F_s)=\chi(\C^{*})=0$, which gives $3-\nu=\sigma$.

(iv) Suppose $\bar E$ has at least four cusps. Let $p\:X_0\to \PP^1$ be the $\PP^1$-fibration induced by $\xi$. Note that every $C_i$, $i=1,\ldots,c$ is tangent to $E_0$. If $E_0$ is vertical for $p$ then, since all fibers of $\PP^1$-fibrations are snc-divisors, all $C_i$ are horizontal, so $h\geq c\geq 4$, which is false. Similarly, if $E_0$ is a section of $p$ then, all $C_i$, being tangent to $E_0$, are horizontal for $p$, and we get a contradiction the same way. Thus $E_0$ is either a $2$- or a $3$-section of $p$. No fiber of $p$ can be contained in $D_0$, because such a fiber would be contained in $D_0-E_0$, which is negative definite. It follows that if $E_0$ is a $3$-section then we get $\#D_0=\rho(X_0)=2+\sum (\#F-1)\geq 2$, where the sum is taken over all singular fibers of $p$, which gives $\#D_0\geq 2+\#(D_0-E_0)=\#D_0+1$; a contradiction. Therefore, $E_0$ is a $2$-section of $p$, so $p_{|E_0}$ is a $2:1$ covering of $\PP^1$ by $\PP^1$. Such a cover has exactly two ramification points. But now at least $c-1\geq 3$ of the $C_i$'s are horizontal and since they are tangent to $E_0$, the corresponding intersection points with $E_0$ are ramification points for $p_{|E_0}$; a contradiction.
\end{proof}

\vskip 0.3cm

\begin{proof}[Proof of Corollary \ref{thm:MAIN3} (part 2)]
We may assume that $\PP^2\setminus\bar E$ has a structural $\C^{**}$-fibration as above. Let $(X,D)\to (\PP^2,\bar E)$ be the minimal log resolution. Recall that $c_0$ and $c_1$ are the numbers of semi-ordinary and non-semi-ordinary cusps and that $c_0'$ denotes the number of ordinary cusps. Put $c_0''=c_0-c_0'$. Clearly, $c=c_0'+c_0''+c_1$. We have $\vc(D_{Y_n})\leq \#D_{Y_n}=\rho(Y_n)+n$ and, by the definition of $\Delta_n$ and $Y_n$, $$\vc(D_n)-\vc(D_{Y_n})\leq b_0(\Delta_n^-)+b_0(\Delta_n^+)+\#\Upsilon_n=b_0(\Delta_n^-)+2(c_0+\eta)-c_0'.$$ We have $\vc(D_0)-\vc(D_n)\leq 2n$, so we get $$\vc(D_0)\leq 2c_0-c_0'+2\eta+b_0(\Delta_n^-)+\rho(Y_n)+3n.$$ For cores we have $\#\core(D_{Y_n})\leq \#D_{Y_n}=\rho(Y_n)+n$, $\#\core(D_n)-\#\core(D_{Y_n})\leq c_0+\eta$ and $\#\core(D_0)\leq \#\core(D_n)$, hence $$\#\core(D_0)\leq n+\eta+c_0+\rho(Y_n).$$

\begin{claim2} $\vc(D)-\vc(D_0)\leq 2(c_1+c_0')+c_0''\text{\ \ and\ \ } \#\core(D)-\#\core(D_0)\leq c_1$. \end{claim2}

This is the part where we use the $\C^{**}$-fibration. It is clear that the contribution from a semi-ordinary cusp to the left hand side of the second inequality is zero and to the first inequality is $2$ or $1$ depending whether the cusp is ordinary or not. It is therefore enough to show that resolving the unique non-snc point in $D_0$ over each non-ordinary cusp adds at most two vertices to the core graph and at most one component to the core.

Let $p\in (D_0-E_0)\cap E_0$ be the non-snc point over some cusp $q_j$ and let $A$ be the component of $D_0-E_0$ tangent to $E_0$. We have $A\cdot E_0=\tau_j$. Let $V$ be the part of $\Exc \psi_0$ over $q_j$. Then $V=[1,(2)_{\tau_j-1}]$. If $A$ is a unique component of $D_0-E_0$ meeting $E_0$ at $p$ or if $\tau_j=2$ then we are done. We may therefore assume that $\tau_j\geq 3$ and that there is a component $B$ of $D_0-E_0-A$ meeting $E_0$ at $p$. By the definition of $\psi_0$ no other component of $D_0$ goes through $p$ and $B\cdot E_0=1$. Since $A$, $B$, $E_0$ go through a non-snc point of $D_0$, they are not contracted by $\alpha_n\circ \psi$ and their intersections at $p$ are not affected by this morphism. Let $A'$, $B'$ and $E'$ be their images in $D_{Y_n}$ and let $f$ be the fiber containing the image of $p$. Denote the divisor of horizontal components of $D_{Y_n}$ by $H$. We have $H\cdot f=3$, so if none of $A'$, $B'$, $E'$ is vertical then $f\cap (Y\setminus D_Y)\cong \C^1$, which contradicts \ref{lem:cuspidal_of_gt_and_khalf}(iii). If $A'=f$ or $E'=f$ then $H\cdot f\geq \tau_j+1\geq 4$; a contradiction. Thus $B'=f$. Let $p'$ be the point of intersection of $H$ and $B'$ different than $p$. We have $(H-E'-A')\cdot B'=1$, so the intersection at $p'$ is transversal. Because $B^2<0=(B')^2$, we see that $\psi$ touches $B$ at $p'$. It follows that there is a twig $T$ of $D_0$ (possibly empty) which meets $B$ in its tip, is contracted by $\psi$ and such that $B\cdot (D_0-E_0-A-T)=0$. Thus the $(-2)$-curves of $V$ are contained in a twig of $D$. The claim follows.

\vskip 0.3cm By \ref{prop:(Xi,Di)_properties} $\eta\leq n$ and by \ref{lem:Cstst_fibrations} $b_0(\Delta_n^-)\leq \sigma$ and $\sigma+n=h+1\leq 4$. Because $c\leq 3$ by \ref{lem:Cstst_fibrations}(iv), we obtain $\vc(D)\leq 2c+c_0+5n+2+b_0(\Delta_n^-)\leq 3c+4n+6\leq 9+16+6=31$ and, since $\rho(Y_n)=2$, $\#\core(D)\leq2n+c+2\leq 13$. This finishes the proof of Corollary \ref{thm:MAIN3}.

\end{proof}

\bibliographystyle{amsalpha}
%\bibliography{bibl}
\bibliography{C:/KAROL/PRACA/PUBLIKACJE/BIBL/bibl2}
\end{document}